\documentclass[11pt]{amsart}
\usepackage{amsbsy}
\usepackage{graphicx,epsfig,subfigure,psfrag}
\usepackage{color}
\begin{document}

\newtheorem{theorem}{Theorem}
\newtheorem{proposition}{Proposition}
\newtheorem{lemma}{Lemma}
\newtheorem{corollary}{Corollary}
\newtheorem{definition}{Definition}
\newtheorem{remark}{Remark}
\newcommand{\tex}{\textstyle}
\numberwithin{equation}{section} \numberwithin{theorem}{section}
\numberwithin{proposition}{section} \numberwithin{lemma}{section}
\numberwithin{corollary}{section}
\numberwithin{definition}{section} \numberwithin{remark}{section}
\newcommand{\ren}{\mathbb{R}^N}
\newcommand{\re}{\mathbb{R}}
\newcommand{\n}{\nabla}
\newcommand{\p}{\partial}
\newcommand{\iy}{\infty}
\newcommand{\pa}{\partial}
\newcommand{\fp}{\noindent}
\newcommand{\ms}{\medskip\vskip-.1cm}
\newcommand{\mpb}{\medskip}
\newcommand{\AAA}{{\bf A}}
\newcommand{\BB}{{\bf B}}
\newcommand{\CC}{{\bf C}}
\newcommand{\DD}{{\bf D}}
\newcommand{\EE}{{\bf E}}
\newcommand{\FF}{{\bf F}}
\newcommand{\GG}{{\bf G}}
\newcommand{\oo}{{\mathbf \omega}}
\newcommand{\Am}{{\bf A}_{2m}}
\newcommand{\CCC}{{\mathbf  C}}
\newcommand{\II}{{\mathrm{Im}}\,}
\newcommand{\RR}{{\mathrm{Re}}\,}
\newcommand{\eee}{{\mathrm  e}}
\newcommand{\LL}{L^2_\rho(\ren)}
\newcommand{\LLL}{L^2_{\rho^*}(\ren)}
\renewcommand{\a}{\alpha}
\renewcommand{\b}{\beta}
\newcommand{\g}{\gamma}
\newcommand{\G}{\Gamma}
\renewcommand{\d}{\delta}
\newcommand{\D}{\Delta}
\newcommand{\e}{\varepsilon}
\newcommand{\var}{\varphi}
\newcommand{\lll}{\l}
\renewcommand{\l}{\lambda}
\renewcommand{\o}{\omega}
\renewcommand{\O}{\Omega}
\newcommand{\s}{\sigma}
\renewcommand{\t}{\tau}
\renewcommand{\th}{\theta}
\newcommand{\z}{\zeta}
\newcommand{\wx}{\widetilde x}
\newcommand{\wt}{\widetilde t}
\newcommand{\noi}{\noindent}
\newcommand{\uu}{{\bf u}}
\newcommand{\xx}{{\bf x}}
\newcommand{\yy}{{\bf y}}
\newcommand{\zz}{{\bf z}}
\newcommand{\aaa}{{\bf a}}
\newcommand{\cc}{{\bf c}}
\newcommand{\jj}{{\bf j}}
\newcommand{\ggg}{{\bf g}}
\newcommand{\UU}{{\bf U}}
\newcommand{\YY}{{\bf Y}}
\newcommand{\HH}{{\bf H}}
\newcommand{\GGG}{{\bf G}}
\newcommand{\VV}{{\bf V}}
\newcommand{\ww}{{\bf w}}
\newcommand{\vv}{{\bf v}}
\newcommand{\hh}{{\bf h}}
\newcommand{\di}{{\rm div}\,}
\newcommand{\ii}{{\rm i}\,}
\def\I{{\rm Id}}
\newcommand{\inA}{\quad \mbox{in} \quad \ren \times \re_+}
\newcommand{\inB}{\quad \mbox{in} \quad}
\newcommand{\inC}{\quad \mbox{in} \quad \re \times \re_+}
\newcommand{\inD}{\quad \mbox{in} \quad \re}
\newcommand{\forA}{\quad \mbox{for} \quad}
\newcommand{\whereA}{,\quad \mbox{where} \quad}
\newcommand{\asA}{\quad \mbox{as} \quad}
\newcommand{\andA}{\quad \mbox{and} \quad}
\newcommand{\withA}{,\quad \mbox{with} \quad}
\newcommand{\orA}{,\quad \mbox{or} \quad}
\newcommand{\atA}{\quad \mbox{at} \quad}
\newcommand{\onA}{\quad \mbox{on} \quad}
\newcommand{\ef}{\eqref}
\newcommand{\mc}{\mathcal}
\newcommand{\mf}{\mathfrak}

\newcommand{\ssk}{\smallskip}
\newcommand{\LongA}{\quad \Longrightarrow \quad}
\def\com#1{\fbox{\parbox{6in}{\texttt{#1}}}}
\def\N{{\mathbb N}}
\def\A{{\cal A}}
\newcommand{\de}{\,d}
\newcommand{\eps}{\varepsilon}
\newcommand{\be}{\begin{equation}}
\newcommand{\ee}{\end{equation}}
\newcommand{\spt}{{\mbox spt}}
\newcommand{\ind}{{\mbox ind}}
\newcommand{\supp}{{\mbox supp}}
\newcommand{\dip}{\displaystyle}
\newcommand{\prt}{\partial}
\renewcommand{\theequation}{\thesection.\arabic{equation}}
\renewcommand{\baselinestretch}{1.1}
\newcommand{\Dm}{(-\D)^m}

\title
{\bf Asymptotic behaviour for a class of quasilinear cooperative eigenvalue problems}

\author{Pablo~\'Alvarez-Caudevilla}

\address{Universidad Carlos III de Madrid,
Av. Universidad 30, 28911-Legan\'es, Spain}
\email{pacaudev@math.uc3m.es}


\thanks{The work has been partially supported by the Ministry of Economy and Competitiveness of
Spain under research project MTM2012-33258.}

\date{\today}





\begin{abstract}
This work is devoted to the analysis of the asymptotic behaviour of a parameter dependent quasilinear 
cooperative eigenvalue system when a parameter in front of some non-negative potentials goes to infinity. In particular we consider operators of $p$-Laplacian type. We prove that the eigenfunctions 
concentrate on the subdomains where those potentials vanish at the limit, while the eigenvalue approaches to an upper bound that will depend on those 
subdomains as well. 
We also show several properties for the unusual limiting problems obtained here.

\end{abstract}

\maketitle


\section{Introduction}
 \label{S1}
 

\subsection{Models and preliminaries}


\noindent In this paper we study the asymptotic behaviour of the following elliptic cooperative eigenvalue problem with potential terms
\begin{equation}
\label{a1}
 \left\{\begin{array}{l}
  (-\Delta_p+\lambda a |u|^{p-2})u-b |u|^{\alpha-1}u\,|v|^\beta v=\tau |u|^{p-2} u,\\
  (-\Delta_q +\lambda d |v|^{q-2} )v- c |u|^{\alpha} u\, |v|^{\beta-1} v=\tau |v|^{q-2} v,
  \end{array}
  \right.
\end{equation}
for $ (u,v) \in W^{1,p}_0(\Omega)\times W^{1,q}_0(\Omega)$ and when the parameter $\lambda\in\mathbb{R}$ goes to infinity. Moreover, $\Omega$ is a smooth bounded domain of $\mathbb{R}^N$, $N\geq 1$, with smooth boundary 
$\partial \Omega$, for example of class $C^2$ or Lipschitz. Also, we consider
$\alpha,\beta \geq 0$ and $p,q>1$ satisfying the relation
\begin{equation}
\label{pq}
\frac{\alpha+1}{p} +\frac{\beta+1}{q}=1.
\end{equation} 
Moreover, the operator $-\Delta_p$  stands for the so called $p$-Laplacian operator so that
$$-\Delta_p u= -{\rm div}\,  (|\nabla u|^{p-2} \nabla u),$$
and similarly for $-\Delta_q$ as a $q$-Laplacian operator. 
To simplify the notation we might write the eigenvalue problem \eqref{a1} in matrix form
\begin{equation}
\label{a5}
  \mathcal{S}(V_1,V_2)\binom{u}{v} :=
  \left( \begin{matrix} (-\Delta_p +V_1)u & - b |u|^{\alpha-1}u\,|v|^\beta v \\ - c |u|^{\alpha} u\, |v|^{\beta-1} v & (-\Delta_q+V_2)v
  \end{matrix}\right)=\tau \binom{|u|^{p-2} u}{ |v|^{q-2} v}.
\end{equation}
In particular, for $V_1= \lambda a |u|^{p-2}$ and $V_2=\lambda d |v|^{q-2}$ we denote $ \mathcal{S}(V_1,V_2)$ as
\begin{equation}
\label{simop}
\mathcal{S}_\lambda:=\mathcal{S}(\lambda a |u|^{p-2},\lambda d |v|^{q-2}) \quad \hbox{and}\quad \mathcal{S}=\mathcal{S}(0,0).
\end{equation}
The exponent of the $p$-Laplacian (respectively the $q$-Laplacian) and the rest of the terms in each equation  
must be consistent so that relation \eqref{pq} is satisfied. Otherwise the system does not have the proper structure of an eigenvalue problem (further details in \cite{LeRa}).

Furthermore, we denote by 
 \begin{equation}
 \label{fireq}
 \tau_1[ \mathcal{S}(V_1,V_2),\Omega],
\end{equation}
the principal eigenvalue of the operator \eqref{a5} under homogeneous Dirichlet boundary conditions and $(\varphi_1,\psi_1)^T$ the principal eigenfunction associated with the principal eigenvalue denoted
by \eqref{fireq}. We will show that such a principal eigenvalue corresponds to the one with an associated positive eigenfunction $(\varphi_1,\psi_1)^T$, 
in the sense that $\varphi_1\geq 0$, $\psi_1\geq 0$ and $\varphi_1\neq 0$, $\psi_1\neq 0$. Indeed, the positive eigenfunction will belong to the positive cone $\mathcal{K}$ with 
positive functions in the interior of the domain $\Omega$ and strictly negative outward normal derivatives on the boundary, i.e.
$$\mathcal{K}:=\{\varphi\in W^{1,p}(\Omega)\,:\, \varphi(x)>0\;\hbox{for $x\in \Omega$},\; \frac{\partial \varphi(x)}{\partial {\bf n}}<0 \;\hbox{for $x$ on $\partial\Omega$} \},$$
where ${\bf n}$ stands for the unitary outward normal vector to the boundary $\partial\Omega$.   
In particular, for the eigenvalue problem \eqref{a1} we have that
$$\tau_1(\lambda)= \tau_1[\mathcal{S}(\lambda a |u|^{p-2},\lambda d |v|^{q-2}),\Omega],$$
as the principal eigenvalue under homogeneous Dirichlet boundary conditions, associated with the principal eigenfunction $(\varphi_\lambda,\psi_\lambda)^T$.

Thus, we are interested in understanding the limit
$$\lim_{\lambda\to \infty} \tau_1(\lambda),$$
under certain assumptions established for the problem \eqref{a1}. In particular, we will assume that $a, d$ are potentials as defined in \cite{PabloAntoine}, i.e.
as Borel functions 
$$a,d: \overline{\Omega}\to \mathbb{R}^+\quad \hbox{such that}\quad a,d \in \mathcal{Q}(\Omega),$$ 
with $\mathcal{Q}(\Omega)$ representing the sets of potentials
so that the following two properties hold:
\begin{enumerate}
\item[1)] $\sup_{x \in \overline{\Omega}}a(x)< + \infty$, $\sup_{x \in \overline{\Omega}}d(x)< + \infty$;
\item[2)] $W_0^{1,p}(\Omega_0^g)=\{u \in W_0^{1,p}(\Omega); \;u=0 \;\text{ {\rm a.e. on } } \Omega_+^g \}$, with $g=a,d$,
\end{enumerate}
Moreover, those potentials $a$ and $d$ are going to be non-negative so that, $a\geq 0$, $d\geq 0$ and we denote the open sets/subdomains of $\Omega$
 where the potentials $a$ and $d$ vanish, as 
\begin{equation}
\label{vanset}
\Omega_0^a := \{x \in  \Omega\; ;\;a(x)=0\}, \quad \quad \Omega_0^d := \{x \in \Omega\; ;\; d(x)=0\}.
\end{equation}
Also, we denote
$$\Omega_0 :=\Omega^{a+d}_0=\Omega_0^a \cap \Omega_0^d = \{x \in \Omega \; ;\; a(x)=d(x)=0\},$$
so that the following subdomains come into play
$$ \Omega_+^a := \{x \in \Omega\;;\;  a(x)>0\}=\overline{\Omega} \backslash \Omega_0^a, \quad
 \Omega_+^d := \{x \in \Omega\;;\;  d(x)>0\}=\overline{\Omega} \backslash \Omega_0^d,$$
 and, there exists a subdomain of $\Omega$ 
 $$\Omega_+=\{x \in \Omega\;;\;  a(x)+d(x)>0\},$$
 which is also an open set, with $\overline\Omega_+\subset \Omega$. Then, we consider a compact set 
 $$K_0=(a+d)^{-1}(0)=\overline\Omega\setminus \Omega_+.$$
 Note that, the compact set $K_0$ consists of two compact subsets in $\mathbb{R}^N$ so that
\begin{equation}
\label{noem}
\Omega_0^a={\rm Int}(K_0^a)\not = \emptyset, \quad \Omega_0^d={\rm Int}(K_0^d)\not = \emptyset,
\end{equation}
and the subdomains $\Omega_0^a$ and $\Omega_0^d$ may not be connected, but with a finite number of components. However, we suppose that 
they are Lipschitz domains on the boundary. Maybe these regularity assumptions on the boundary might be relaxed with extra care on the
analysis but for our purposes Lipschitz boundary conditions are enough.

We also suppose that  $b, c \in C(\overline{\Omega})$ and
\eqref{a1} is \emph{strongly cooperative} in the sense that
\begin{equation*}
    b(x)>0 \quad \hbox{and}\quad c(x)>0,
    \quad \hbox{for all}\;\; x \in \overline \Omega.
\end{equation*}


\subsection{Motivation}


\noindent Most of previous results rely on the variational structure of the problem, for example \cite{KaLind}, \cite{Lind}. 

Furthermore, focusing on cooperative systems, there are previous works analysing similar quasilinear systems to \eqref{a1}
(see for example \cite{LeRa} and references therein) where the proofs were based on variational techniques and assuming 
the particular case when the cooperative terms are equal, i.e. $b=c$ in \eqref{a1} multiplied by the proper coefficients depending on the exponents of the coupling terms. 
Indeed, assuming that those off-diagonal couple terms are the same having the proper variational structure, there are several additional works. We would like to stress the work done by
Bozhkov \& Mitidieri \cite{BMit} for a variational system of the form 
\begin{equation}
\label{varmit}
\left\{
\begin{array}{ll}
 \begin{array}{l}
  -\Delta_p u= (\alpha+1)c(x) |u|^{\alpha-1}u\,|v|^\beta v+\lambda a(x) |u|^{p-2} u,\\
  -\Delta_q v= (\beta+1) c(x) |u|^{\alpha} u\, |v|^{\beta-1} v +\mu b(x) |v|^{q-2} v,
  \end{array}
  &  \hbox{in}\quad \Omega,\\ 
  \;\;\;(u,v)=(0,0), & \hbox{on}\quad \partial \Omega,
\end{array} \right.
\end{equation}
where some results of existence and multiplicity of solutions were obtained via the fibering method. This methodology was introduced
by S.I.~Pohozaev in the 1970s \cite{Poh0, PohFM}, as a convenient
generalization of previous versions by Clark \cite{Clark} and Rabinowitz
\cite{Rabin}  of variational approaches, and further
developed by Dr\'abek and Pohozaev \cite{DraPoh} and others in
the 1980's to ascertain the existence
and multiplicity of solutions for equations with a variational
form (in particular and relevant for our work, the $p$-Laplacian) associated with such equation,
i.e. potential operator equations, alternatively with other methods
such as bifurcation theory, critical point theory and so on.

However, for non-variational systems such as \eqref{a1} there are only a few results involving $pq$-Laplacian operators. In this case we would like to emphasise 
the work done by 
Cl\'ement, Fleckinger, Mitidieri and Th\'elin \cite{CFMT} where they proved the existence of radial positive solutions for a 
quasilinear elliptic system of the form 
 $$\left\{
\begin{array}{ll}
 \begin{array}{l}
  -\Delta_p u= |u|^{\alpha}\,|v|^\beta,\\
  -\Delta_q v= |u|^{\gamma} \, |v|^{\delta},
  \end{array}
  &  \hbox{in}\quad \Omega,\\ 
  \;\;\; (u,v)=(0,0), & \hbox{on}\quad \partial \Omega,
\end{array}\right.
$$
where all the exponents $\alpha,\beta,\delta,\gamma$ are non-negative and imposing the condition $\beta\gamma>0$, which maintains the system coupled.

As one of the main motivations to study these kinds of problems, system \eqref{a1} is an extension of a simiilar system assuming the Laplacian operators, 
which represents the steady states of a parabolic cooperative problem of the form
$$\left\{ \begin{array}{ll} \begin{array}{l}
    \frac{\partial u}{\partial t}-\Delta u=\lambda u + b(x) v - a(x) f(x,u)u\\
  \frac{\partial v}{\partial t}-\Delta v= c(x) u+ \lambda v  - d(x) g(x,v)v
  \end{array} & \hbox{in }\;\; \Omega\times (0,\infty),\\
 \; (u,v)^T = (0,0)^T & \hbox{on }\;\;\partial\Omega\times(0,\infty),\\
  \;(u(\cdot,0),v(\cdot,0))^T =
  (u_0,v_0)^T> (0,0)^T & \hbox{in }\;\; \Omega,\end{array}\right.
  $$
  where $\gamma \in \mathcal{C}^\mu(\overline\Omega)$, for all $\gamma \in
\{a,b,c,d\}$ and some $\mu\in (0,1]$, such that $a,d\geq 0$ and $b,c>0$ in $\Omega$,  $\lambda$ is a real parameter and the nonlinear functions $f,g \in \mathcal{C}^{\mu,1+\mu}(\bar\Omega\times
[0,\infty))$ satisfy that for any $h\in\{f,g\}$, $h(x,0)=0$ and $\partial_u h(x,u)>0$ for all $x \in\overline\Omega$ and $u>0$. This model 
arises in population dynamics 
for the analysis of cooperative species where the environments are heterogeneous (see \cite{AL-NA} for further details). 
In those parabolic models the principal eigenvalue of the linearised associated elliptic problem seems 
to be crucial in ascertaining the long time behaviour of the population, i.e. the linear eigenvalue problem
\begin{equation}
\label{lineg}
\begin{array}{ll}
 \left\{\begin{array}{l}
  -\Delta u +\lambda a u-b v=\tau u,\\
  -\Delta v +\lambda d v- c u=\tau v,
  \end{array}
  \right.\; 
  &  \hbox{for}\; (u,v)^T \in H^{1}_0(\Omega)\times H^{1}_0(\Omega),
\end{array}
\end{equation}
Indeed, there are several works analysing 
that asymptotic behaviour, when $\lambda$ goes to infinity, assuming spatial heterogeneities for the potentials in front of the parameter $\lambda$, for one single equation such as \cite{PJ}, and for
cooperative systems as well such as \cite{AC-LG}, \cite{Dan2}, or for very general spatial heterogeneities conditions \cite{PabloAntoine}.

Thus, the problem under consideration here might represent the steady state solutions of the parabolic problem
\begin{equation}
\label{parab}
 \left\{ \begin{array}{ll} \begin{array}{l}
    \frac{\partial u}{\partial t}+ (-\Delta_p+\lambda a |u|^{p-2})u-b |u|^{\alpha-1}u\,|v|^\beta v=\tau |u|^{p-2} u\\
  \frac{\partial v}{\partial t}+(-\Delta_q +\lambda d |v|^{q-2} )v- c |u|^{\alpha} u\, |v|^{\beta-1} v=\tau |v|^{q-2} v
  \end{array} & \hbox{in }\;\; \Omega\times (0,\infty),\\
  \;(u,v)^T = (0,0)^T & \hbox{on }\;\;\partial\Omega \times(0,\infty),\\
 \; (u(\cdot,0),v(\cdot,0))^T =
  (u_0,v_0)^T> (0,0)^T & \hbox{in }\;\; \Omega,\end{array}\right.
\end{equation}
where, again, the principal eigenvalue of the elliptic problem \eqref{a1} will play an important role in the dynamical behaviour of the model. 
In relation to the convergence of similar problems, however assuming fractional operators, in \cite{Lin-Zheng} 
the authors obtained some multiplicity results and the convergence of those solutions when a parameter is passed to the limit to a limiting problem with infinitely many solutions. 



\subsection{Main results}


\noindent Here, as mentioned above, we assume a non-variational cooperative system \eqref{a1} assuming that the off-diagonal terms $b$ and $c$ of the $pq$-Laplacian system \eqref{a1} are different, so 
considering a non-self-adjoint operator of the form \eqref{a5} with 
$$b(x)\neq c(x), \quad \hbox{for any}\quad x\in\overline \Omega,$$
and, then, obtaining 
the convergence of the principal eigenvalue under the very general heterogenous assumptions $1)$ and $2)$ 
for the non-negative potentials $a$ and $d$ set up above, and under a less restrictive system as the one assumed in \cite{LeRa}. In \cite{LeRa} the authors obtained such a convergence for a system 
with the same cooperative coefficients, applying variational arguments
and under several constraints. Consequently, the analysis shown here is completely different from the previous works.

Although, our methods will depend on certain convergences, other important elements 
are based on operator theory that will provide us with several important properties of the principal eigenvalue of our problem. 

Indeed, these properties will play a crucial role in obtainning the 
asymptotic behaviour of problem \eqref{a1}, as well as having their own interest since they have never been proved before. Especially due to the non-variational nature of problem 
\eqref{a1}. 
Indeed, despite the fact that our problem lacks such a variational structure we are able to prove the following crucial issue: the principal eigenvalue is the smallest, simple, positive eigenvalue and, also, isolated. To do so we use different 
approaches for linear problems that we adapt for the nonlinear system \eqref{a1} as well as 
some arguments shown in \cite{GarSabTak}. 
Note that those properties for the eigenvalue problem \eqref{a1} might be established as well for the unusual limiting problem obtained here; see details below.

Consequently, under the assumptions established above we state the main result of this paper.

\begin{theorem} \label{stabb} 
Let $\Omega \subset \mathbb{R}^N$ be an open set and assume that $a$ and  $d$ are two potentials for which conditions $1)$ and $2)$ are fulfilled. Then,
\begin{eqnarray}
\lim_{\lambda \to +\infty}\tau_1(\lambda) = \tau_1 , \label{convergence}
\end{eqnarray}
 where $\tau_1(\lambda)$ is the principal eigenvalue of the $pq$-Laplacian system \eqref{a1}, and the limiting principal eigenvalue 
 $\tau_1$ is the one corresponding to the limiting $pq$-Laplacian system 
 \begin{equation}
\label{a1pr}
 \left\{\begin{array}{l}
  -\Delta_p u-b |u|^{\alpha-1}u\,|\mathcal{P}v|^\beta \mathcal{P}v=\tau_1 |u|^{p-2} u,\\
  -\Delta_q v- c |\mathcal{P}u|^{\alpha} \mathcal{P}u\, |v|^{\beta-1} v=\tau_1 |v|^{q-2} v,
  \end{array}
  \right.
\end{equation}
for $(u,v)^T \in W^{1,p}_0(\Omega_0^a)\times W^{1,q}_0(\Omega_0^d)$ and where $\mathcal{P}$ stands for the projection on the subdomain where both potentials vanish at the same time, i.e.
$$\mathcal{P}w=\chi_{\Omega_0^{a}\cap \Omega_0^{d}}w, \quad \hbox{with}\quad w=u,v.$$
In fact, if the intersection of the subdomains where the potentials $a$ and $d$ vanish were empty 
$$\Omega_0^{a}\cap\Omega_0^{d}=\emptyset,$$
$\mathcal{P}$ is defined to be zero and
the limiting principal eigenvalue $\tau_1$ will be the infimum among the principal eigenvalues corresponding to the uncoupled system
 \begin{equation}
 \label{uncop}
\left\{
\begin{array}{ll}
  -\Delta_p u=\tau_1 |u|^{p-2} u & \hbox{in}\quad \Omega_0^{a},\\
  -\Delta_q v=\tau_1 |v|^{q-2} v & \hbox{in}\quad \Omega_0^{d},
  \end{array}
    \right.
    \end{equation}
    under homogeneous Dirichlet boundary conditions and such that
    $$\tau_1=\inf\{\tau_1[ -\Delta_p,\Omega_0^{a}],\tau_1[ -\Delta_q,\Omega_0^{d}]\}.$$
 In addition, any sequence of normalised eigenfunctions  $\{(\varphi_\lambda,\psi_\lambda)^T\}$ associated with $\tau_1(\lambda)$ admits a subsequence that converges strongly in 
 $W^{1,p}_0(\Omega)\times W^{1,q}_0(\Omega)$ to the normalised eigenfunction $(\varphi_*,\psi_*)^T$ in $L^p(\Omega)\times L^q(\Omega)$ associated with $\tau_{1}$,
in the sense that
$$\int_\Omega |\varphi_*|^p + \int_\Omega |\psi_*|^q =1.$$  
\end{theorem}

\begin{remark}
\rm{The principal eigenvalue $\tau_1(\l)$ is simple, isolated and positive (see details below) so the whole sequence $\{(\varphi_\lambda,\psi_\lambda)^T\}$ converges strongly in 
$W^{1,p}_0(\Omega)\times W^{1,q}_0(\Omega)$  to the corresponding normalised eigenfunction $(\varphi_*,\psi_*)^T \in W^{1,p}_0(\Omega_0^a)\times W^{1,q}_0(\Omega_0^d)$
(e.g. \cite{AC-LG}). We will prove this fact as part of Theorem\;\ref{stabb}.

Moreover, we cannot forget that for this quasilinear system \eqref{a1} we have the possibility of existence of semi-trivial solutions.} 
\end{remark}

We define the spaces $W_0^{1,p}(E)$ 
for any measurable set $E \subseteq \mathbb{R}^N$ in terms of the capacity. In other words,
\begin{equation}
W_0^{1,p}(E) := \{ u\in W^{1,p}(\mathbb{R}^N) \text{ so that } u=0 \text{ q.e.  on } \mathbb{R}^N \backslash E\}, \label{defH1}
\end{equation}
when $E \subseteq \mathbb{R}^N$ is any measurable subset (not necessarily open). Observe, that particularly
for any open set $\Omega \subset \mathbb{R}^N$ there is a nice characterisation of $W_0^{1,p}(\Omega)$ using capacity 
(see for instance  \cite{bucbut} Theorem 4.1.2. or \cite{Z}), namely we have
$$u \in W_0^{1,p}(\Omega) \Leftrightarrow ( u\in W^{1,p}(\mathbb{R}^N) \text{ and } u=0 \text{ q.e.  on } \mathbb{R}^N \backslash \Omega).$$
Note that $W_0^{1,p}(E)$ is a closed subspace of $W^{1,p}(\mathbb{R}^N)$ and inherits its 
 structure. Consequently, the imbedding $W_0^{1,p}(E)$ into $L^p(E)$ remains compact. 
 Notice also that according to our definition $W_0^{1,p}(E)$ is never empty, because it always contains the function identically equal to $0$. That could be also deduced from the condition 
 \eqref{noem}. 
  It is also clear from the definition that  $W_0^{1,p}(E)=\{0\}$ when the 
capacity of $E$ is zero.

However, when the potentials are continuous functions we must recall that the definition for the space $W_0^{1,p}(\Omega)$ denoted above by \eqref{defH1} is 
equivalent to 
$$W_0^{1,p}(\Omega) := \{ u\in W^{1,p}(\mathbb{R}^N) \text{ so that } u=0 \text{ a.e.  on } \mathbb{R}^N \backslash \Omega\},$$
since the set $\Omega$ will be open and the zero set has capacity zero. This implies that part of the zero set of the potentials $a$ and $d$ with measure zero could be inside 
$\Omega_+^a$ and $\Omega_+^d$ respectively. Thus, under hypothesis \eqref{vanset} we know that 
\begin{align*}
W^{1,p}_0(K_a) & :=W^{1,p}(\mathbb{R}^N)\cap\{u=0 \text{ q.e. in } \mathbb{R}^N\setminus K_a\}\\ 
& \; =W^{1,p}(\mathbb{R}^N)\cap\{u=0 \text{ a.e. in } \mathbb{R}^N\setminus K_a\}.
\end{align*}
Indeed, the only kind of regularity assumption that is contained in condition $2)$ which can be understood as a stability-type property for $W^{1,p}_0$ 
(resp. $W^{1,q}_0$).


\section{$p$--Laplacian eigenvalue problem}


We assume that the Banach space $W^{1,p}_0(\Omega)$ with $1<p <\infty$, for a bounded domain $\Omega$ in $\mathbb{R}^N$, is equipped with the norm 
$$\|u\|_{W^{1,p}_0(\Omega)} := \left( \int_\Omega |\nabla u|^p \right)^{1/p}.$$ 
Thanks to Poincar\'e's inequality this norm is equivalent to the standard one for the Sobolev spaces $W^{1,p}_0(\Omega)$ in bounded domains.
The pair $(u,v)^T \in W^{1,p}_0(\Omega) \times W^{1,q}_0(\Omega)$ is called a weak solution of the problem \eqref{a1} if the following integral identities hold
\begin{equation}
\label{sisin}
\begin{split}
& \int_\Omega |\nabla u|^{p-2} \nabla u \nabla \nu_1 +\lambda  \int_\Omega a |u|^{p-2} u \,\nu_1 - \int_\Omega b  |u|^{\alpha}\,|v|^{\beta+1} \nu_1 = \tau(\lambda) \int_\Omega |u|^{p-2} u\, \nu_1,\\ & 
 \int_\Omega |\nabla v|^{q-2} \nabla v \nabla \nu_2+  \lambda   \int_\Omega d |v|^{q-2} v \,\nu_2 -  \int_\Omega c |u|^{\alpha+1} \, |v|^{\beta} \nu_2 =\tau(\lambda)  \int_\Omega |v|^{q-2} v \, \nu_2,
 \end{split}
 \end{equation}
for any $(\nu_1,\nu_2)^T \in  W^{1,p}_0(\Omega) \times W^{1,q}_0(\Omega)$. In addition, a real number $\tau$ is called an eigenvalue under homogeneous 
Dirichlet boundary conditions and $(u,v)^T$ is 
its associated eigenfunction of the system
\begin{equation}
\label{sys}
\left\{ \begin{array}{ll}
\begin{array}{l}
 -\Delta_p u-b |u|^{\alpha-1}u\,|v|^\beta v=\tau |u|^{p-2} u\\
 -\Delta_q v- c |u|^{\alpha} u\, |v|^{\beta-1} v=\tau |v|^{q-2} v
 \end{array} & \hbox{in}\quad \Omega, \\ 
 \;\; \; (u,v)=(0,0), & \hbox{on} \quad \partial \Omega,
  \end{array} \right.
\end{equation}
if  
{\small \begin{equation}
\label{qeaks}
\begin{split}
\int_\Omega  & |\nabla u|^{p-2}   \nabla u \nabla \nu_1  +  \int_\Omega |\nabla v|^{q-2} \nabla v \nabla \nu_2
 \\ & = \tau \left( \int_\Omega |u|^{p-2} u\, \nu_1+  \int_\Omega |v|^{q-2} v \, \nu_2\right) + \int_\Omega b  |u|^{\alpha}\,|v|^{\beta+1} \nu_1 
+ \int_\Omega c |u|^{\alpha+1} \, |v|^{\beta}  \nu_2,
\end{split}
\end{equation}}
for every $(\nu_1,\nu_2)^T \in  W^{1,p}_0(\Omega) \times W^{1,q}_0(\Omega)$.

\subsection{Operator properties}


Furthermore, we establish some properties for the operators we are working with. 
In particular, we first obtain the semicontinuity and monotonicity of the operators for the single equation \eqref{plapl}. Indeed, assuming the problem
\begin{equation}
\label{exlp}
 \left\{\begin{array}{ll}
  (-\Delta_p +\lambda a |u|^{p-2})u=f, & \hbox{in}\quad \Omega,\\
 \; u=0, &  \hbox{on}\quad \partial\Omega,
\end{array} \right.
\end{equation}
for $p>1$ and $f $ in the dual space of 
$Y=W^{1,p}_0(\Omega)$, denoted by $Y'=W^{-1,p'}(\Omega)$ with $1=\frac{1}{p}+\frac{1}{p'}$ we obtain the following result. 

\begin{lemma}
\label{lepeq}
Let $f\in W^{-1,p'}(\Omega)$ be for the problem \eqref{exlp}. Then, the operator 
\begin{equation}
\label{opr}
 \mathcal{B}(\cdot)=-\Delta_p(\cdot) +\lambda a |\cdot|^{p-2}(\cdot),
 \end{equation}
 is continuous and monotone. 
 \end{lemma}
 \begin{proof}
 Take a bounded sequence $\{u_n\}$ in $W^{1,p}_0(\Omega)$. Then, due to compactness
we can find a subsequence, again labelled $\{u_n\}$, which converges weakly in $W^{1,p}_0(\Omega)$ and strongly in $L^p(\Omega)$ to a certain $u$. Moreover, 
{\small \begin{align*}
& \|\mathcal{B}(u_n)  -\mathcal{B}(u)\|_{\mathcal{M}(Y)}  =\sup_{\|h\|_Y \leq 1} \left\| \langle \mathcal{B}(u_n)-\mathcal{B}(u),h \rangle \right\|_{Y}\\ & 
= \sup_{\|h \|_Y \leq 1} \left\|\int_\Omega [|\nabla u_n|^{p-2} \nabla u_n - |\nabla u|^{p-2} \nabla u] \cdot \nabla h  
+\lambda  \int_\Omega a [|u_n|^{p-2} u_n - |u|^{p-2} u] h \right\|_{Y},
\end{align*}}
with $\mathcal{M}(Y)$ denoting the space of bounded operators from $Y$ to $Y$. And, applying H\"older's inequality yields
{\small \begin{equation}
\label{opno}
\begin{split}
& \| \mathcal{B}(u_n)    -\mathcal{B}(u)\|_{\mathcal{M}(Y)}  \leq \left(\int_\Omega  \left| |\nabla u_n|^{p-2} \nabla u_n - |\nabla u|^{p-2} \nabla u\right|^{p'}\right)^{1/p'} 
\left(\int_\Omega  \left|\nabla h \right|^{p}\right)^{1/p} 
\\ & +\lambda \left(\int_\Omega  \left| a [|u_n|^{p-2} u_n - |u|^{p-2} u] \right|^{p'}\right)^{1/p'} \left(\int_\Omega  \left| h \right|^{p}\right)^{1/p}
\\ & 
\leq K \left[ \left(\int_\Omega  \left| |\nabla u_n|^{p-2} \nabla u_n - |\nabla u|^{p-2} \nabla u\right|^{p'}\right)^{1/p'}  + 
\l \left(\int_\Omega  \left| a [|u_n|^{p-2} u_n - |u|^{p-2} u] \right|^{p'}\right)^{1/p'}\right].
\end{split}
\end{equation}}
Note that 
$$\frac{1}{p}+\frac{1}{p'}=1 \quad\hbox{so that,}\quad p'=\frac{p}{p-1}.$$ 
Thus, we find that \eqref{opno} certainly converges to zero because of the convergence of the sequence $\{u_n\}$ in $L^p(\Omega)$ and, also, by  the continuity of the 
Nemytski operator 
$$M(s)=|s|^{p-2}s, \quad\hbox{from $L^p(\Omega)$ to $L^{p'}(\Omega)$}.$$
Now, we see that the operator $\mathcal{B}$ is also a monotone operator. Indeed, for $u\neq v$
\begin{align*}
 \langle & \mathcal{B}(u)- \mathcal{B}(v),u-v \rangle  =  \int_\Omega [|\nabla u|^{p-2} \nabla u - |\nabla v|^{p-2} \nabla v] \cdot (\nabla u -\nabla v)\\ & 
 +\lambda  \int_\Omega a [|u|^{p-2} u - |v|^{p-2} v] (u-v)\\ &
 \geq \int_\Omega |\nabla u|^{p} -  \int_\Omega  |\nabla v|^{p-2} \nabla v \cdot \nabla u  + \int_\Omega |\nabla v|^{p} -  \int_\Omega |\nabla u|^{p-2} \nabla u\cdot  \nabla v\\ & 
 \geq  \int_\Omega |\nabla u|^{p} -  \left( \int_\Omega |\nabla v|^{p'} \right)^{1/p'} \left( \int_\Omega   |\nabla u|^{p}\right)^{1/p}
  \\ & + \int_\Omega |\nabla v|^{p} -  \left( \int_\Omega |\nabla u|^{p'} \right)^{1/p'} \left( \int_\Omega  |\nabla v|^{p}\right)^{1/p}
 \\ & = \left[ \| u\|^{p-1} - \|v\|^{p-1} \right] \left[\|u\|-\|v\|\right] \geq 0,
\end{align*}
since $s\mapsto |s|^{p-1}$ is an increasing function on $(0,+\infty)$. Therefore, the operator $\mathcal{B}$ is monotone and continuous. 
\end{proof}

\begin{remark}
Consequently, by Lemma\;\ref{lepeq} we have the existence of solutions for a problem 
of the form \eqref{exlp} from Browder's Theorem \cite[Th. 5.3.22]{DraMil}, since the operator is also semicontinuous, bounded and coercive. 
\end{remark}

\vspace{0.2cm}

The conclusions established above for the operators $ \mathcal{B}$ \eqref{opr} corresponding to the single equation \eqref{exlp} 
can be extended to a $pq$-Laplacian cooperative system of the form 
\begin{equation}
\label{coopps}
\left\{\begin{array} {ll} \begin{array}{l}  (-\Delta_p+\lambda a |u|^{p-2}) u-b |u|^{\alpha-1}u\,|v|^\beta v=f,\\ (-\Delta_q+\lambda d |v|^{q-2}) 
v- c |u|^{\alpha} u\, |v|^{\beta-1} v=g,\end{array} & \hbox{in}\quad \Omega,\\
\;\;u=v=0, & \hbox{on}\quad \partial\Omega,\end{array}\right.
\end{equation}
for $p>1$, $q>1$. 
Thus, we prove, following similar arguments and under certain extra assumptions, 
that the operator corresponding to the cooperative system is continuous. However, in general, we must stress that for large cooperative terms $b$ and $c$ 
the monotonicity is not true. Hence, only the continuity can be obtained. 
\begin{lemma}
\label{lecoops}
Assume a cooperative system of the form \eqref{coopps} such that $f$ belongs to the dual space of $Y_1=W^{1,p}_0(\Omega)$ denoted by 
$$Y_1'=W^{-1,p'}(\Omega)\quad \hbox{with}\quad 1=\frac{1}{p}+\frac{1}{p'},$$ 
$g$ belongs to the dual space of 
$Y_2=W^{1,q}_0(\Omega)$ denoted by 
$$Y_2'=W^{-1,q'}(\Omega) \quad \hbox{with}\quad 1=\frac{1}{q}+\frac{1}{q'},$$ and condition \eqref{pq} is also satisfied. Then,  the operator 
$\mathcal{S}_\lambda$ acting in \eqref{a5} is continuous. 
\end{lemma}
\begin{proof}
Taking a bounded sequence 
  $$\{(u_n,v_n)^T\} \in W^{1,p}_0(\Omega)\times W^{1,q}_0(\Omega),$$  
  by compactness
we can find a subsequence, again labelled $\{(u_n,v_n)^T\}$, which converges weakly in 
$W^{1,p}_0(\Omega)\times W^{1,q}_0(\Omega)$ and strongly in $L^p(\Omega)\times L^q(\Omega)$ to a certain $(u,v)^T$. Moreover, 
{\small \begin{align*}
& \left\| \mathcal{S}_\lambda \binom{u_n}{v_n}-  \mathcal{S}_\lambda \binom{u}{v}\right\|_{\mathcal{M}(W)}  =\sup_{\begin{smallmatrix} \|h_i\|_{Y_i} \leq 1\\
    i=1,2\end{smallmatrix}} \left\| \langle \mathcal{S}_\lambda \binom{u_n}{v_n}-\mathcal{S}_\lambda\binom{u}{v},\binom{h_1}{h_2} \rangle \right\|_{W}\\ & 
= \sup_{\begin{smallmatrix} \|h_i\|_{Y_i} \leq 1\\
    i=1,2\end{smallmatrix}} \|\int_\Omega [|\nabla u_n|^{p-2} \nabla u_n - |\nabla u|^{p-2} \nabla u] \cdot \nabla h_1 \\ & + 
    \int_\Omega [|\nabla v_n|^{q-2} \nabla v_n - |\nabla v|^{q-2} \nabla v] \cdot\nabla h_2\\ & 
+\lambda \left( \int_\Omega a [|u_n|^{p-2} u_n - |u|^{p-2} u] \,h_1 + \int_\Omega d [|v_n|^{q-2} v_n - |v|^{q-2} v] \,h_2\right)\\ & 
- \int_\Omega b |u|^{\alpha-1}u\,|v|^\beta v\, h_1 - \int_\Omega c |u|^{\alpha} u\, |v|^{\beta-1} v\, h_2\|_{W},
\end{align*}}
where $W:=Y_1\times Y_2= W^{1,p}_0(\Omega)\times W^{1,q}_0(\Omega)$. It is now clear that it converges to zero, once we apply the 
H\"older's inequality and the cooperative character of the system. Indeed, the following holds
{\small \begin{align*}
& \left\| \mathcal{S}_\lambda \binom{u_n}{v_n}-  \mathcal{S}_\lambda \binom{u}{v}\right\|_{\mathcal{M}(W)}  \leq \left(\int_\Omega  \left| 
|\nabla u_n|^{p-2} \nabla u_n - |\nabla u|^{p-2} \nabla u\right|^{p'}\right)^{1/p'} 
\left(\int_\Omega  \left|\nabla h_1 \right|^{p}\right)^{1/p} 
\\ & + \left(\int_\Omega  \left| |\nabla v_n|^{q-2} \nabla v_n - |\nabla v|^{q-2} \nabla v\right|^{q'}\right)^{1/q'} 
\left(\int_\Omega  \left|\nabla h_2 \right|^{q}\right)^{1/q} 
\\ & +\lambda \left(\int_\Omega  \left| a [|u_n|^{p-2} u_n - |u|^{p-2} u] \right|^{p'}\right)^{1/p'} \left(\int_\Omega  \left| h_1 \right|^{p}\right)^{1/p}\\ & 
+\lambda \left(\int_\Omega  \left| d [|v_n|^{q-2} v_n - |v|^{q-2} v] \right|^{q'}\right)^{1/q'} \left(\int_\Omega  \left| h_2 \right|^{q}\right)^{1/q}
\\ & 
\leq K (\left(\int_\Omega  \left| |\nabla u_n|^{p-2} \nabla u_n - |\nabla u|^{p-2} \nabla u\right|^{p'}\right)^{1/p'} + 
\lambda \left(\int_\Omega  \left| a [|u_n|^{p-2} u_n - |u|^{p-2} u] \right|^{p'}\right)^{1/p'}
\\ & \left(\int_\Omega  \left| |\nabla v_n|^{q-2} \nabla v_n - |\nabla v|^{q-2} \nabla v\right|^{q'}\right)^{1/q'} + 
\l \left(\int_\Omega  \left| d [|v_n|^{q-2} v_n - |v|^{q-2} v] \right|^{q'}\right)^{1/q'}).
\end{align*}}
Therefore, we obtain the continuity of the operator $\mathcal{S}_\lambda$. 
\end{proof}

\vspace{0.2cm}


\subsection{Principal eigenvalue properties}


In this section we establish a result which is a counterpart of the already existent one corresponding 
to the single equation (see e.g. \cite{An,Bar,Lind}) of the form 
\begin{equation}
\label{plapl}
	-\Delta_p u=\tau |u|^{p-2} u,
\end{equation}
under homogeneous Dirichlet boundary conditions,  
for which it is known that it admits, a unique positive first eigenvalue $\tau_1$ with a non-negative eigenfunction. 
Moreover, the principal eigenvalue of problem \eqref{plapl} is also isolated and simple as a consequence of its variational characterisation. 

For cooperative systems one can consider the variational cooperative eigenvalue problem of the form 
\begin{equation}
\label{sysvar}
\left\{ \begin{array}{ll}
\begin{array}{l}
 -\Delta_p u-(\alpha+1)B |u|^{\alpha-1}u\,|v|^\beta v=\tau_0 |u|^{p-2} u\\
 -\Delta_q v- (\beta+1)B |u|^{\alpha} u\, |v|^{\beta-1} v=\tau_0 |v|^{q-2} v
 \end{array} & \hbox{in}\quad \Omega, \\ 
  \;\;\,(u,v)=(0,0) & \hbox{on} \quad \partial \Omega,
  \end{array} \right.
\end{equation}
where  $\tau_0=\tau_0[ \mathcal{S}_0,\Omega]$ 
stands for the principal eigenvalue for the 
symmetric operator
$$
  \mathcal{S}_0\binom{u}{v} :=
  \left( \begin{matrix} -\Delta_p u & - (\alpha+1)B |u|^{\alpha-1}u\,|v|^\beta v \\ - (\beta+1)B |u|^{\alpha}u\, |v|^{\beta-1} v & -\Delta_q v
  \end{matrix}\right).
$$
Furthermore, since problem \eqref{sysvar} has a variational 
structure, in this case we have an expression for the first eigenvalue
based on the Rayleigh quotient, i.e.
{\small \begin{equation}
\label{rayta}
\tau_0=\inf_{u\in W^{1,p}_0(\Omega),v\in W^{1,q}_0(\Omega)} \frac{\int_\Omega |\nabla u|^p+ \int_\Omega |\nabla v |^q -(\alpha+\beta+2) \int_\Omega B |u|^{\alpha+1} |v|^{\beta+1}}{\int_\Omega u^p+\int_\Omega v^q}.
\end{equation}}
Hence, applying the results obtained by Kawohl \& Lindqvist \cite{KaLind} and Lindqvist \cite{Lind} we can find that the first 
eigenvalue of the problem \eqref{sysvar} is unique, positive and isolated for any arbitrary domain $\Omega$ in $\mathbb{R}^N$. 
Moreover, its associated eigenfunction is unique and positive in $\Omega$ and with maximal regularity $C^{1,\alpha}(\Omega)$, 
for some $0<\alpha <1$ by elliptic regularity. Those facts 
can be proved following the arguments nicely shown in \cite{KaLind} for the problem \eqref{sysvar}; see also for example \cite[Lemma 2.3]{LeRa}.

Thus, we state a similar result for the cooperative system of the form \eqref{sys} proving that there is no positive eigenvalue below $\tau_1$ 
(denoting the smallest eigenvalue) as well as to being an isolated eigenvalue from above, 
associated with a positive eigenfunction $(\varphi_{1},\psi_{1})^T\in W_0^{1,p}(\Omega)\times W_0^{1,q}(\Omega)$, so that $\varphi_1>0$ and $\psi_1>0$, with negative 
outward normal derivatives on the boundary.  However, we must note that our problem is not variational so a different approach must be followed. The results actually rely on 
the strong maximum principle for $p$-Laplacian problems.

To do so we introduce the following notation
$$ -\Delta_p u-b |u|^{\alpha-1}u\,|v|^\beta v = \mathcal{A}_1 w, \quad  -\Delta_q v- c |u|^{\alpha} u\, |v|^{\beta-1} v=\mathcal{A}_2 w,$$
and
where $w=(u,v)$ and $\mathcal{A}_i$, for $i=1,2$ the operator involved in each equation so that \eqref{sys} becomes
\begin{equation}
\label{gosy}
\left\{ 
\begin{array}{l}
 \mathcal{A}_1 w=\tau (\mathcal{D} w)_1,\\
\mathcal{A}_2 w=\tau (\mathcal{D} w)_2. \end{array} \right.
\end{equation}
with 
$$\tau |u|^{p-2} u =\tau (\mathcal{D} w)_1, \quad \tau |v|^{q-2} v= \tau (\mathcal{D} w)_2.$$
According to Lemma\;\ref{lecoops} 
the operator $\mathcal{S}$ is continuous 
in the positive cone $\mathcal{K}$ corresponding to the Banach space $W_0^{1,p}(\Omega)\times W_0^{1,q}(\Omega)$ and 
from condition \eqref{pq} on the exponents we find that
\begin{equation}
\begin{split}
\label{expcon}
\hbox{either} & \quad p-1<\alpha+\beta+1 < q-1,\\ 
\hbox{or} & \quad q-1 <\alpha+\beta+1<p-1,
\end{split}
\end{equation}
depending on the assumption of $p\leq q$ or $q\leq p$. 
Then assuming \eqref{expcon} we introduce a ``pseudo-homogenous" condition for the operator $\mathcal{S}$ saying that
\begin{equation}
\label{pseu}
\mathcal{S}(t {\bf w})\geq t^{r-1} \mathcal{S}({\bf w}),\quad \hbox{where ${\bf w}=(u,v)^T$ and with}\quad r=p\;\hbox{or}\;q,
\end{equation}
depending on the different possible situations, and if $t \geq 0$. In particular, it follows that
$$ \mathcal{A}_1 (t {\bf w})\geq t^{r-1}  \mathcal{A}_1 {\bf w},\quad  \mathcal{A}_2 (t {\bf w})\geq t^{r-1}  \mathcal{A}_2 {\bf w}.$$
The last property is similar to the equivalent one for linear operators 
$$\mathcal{L}(\mu {\bf w})=\mu \mathcal{L}({\bf w}),$$
which means that the operator is homogeneous of degree one or simply homogeneous. 

Furthermore, thanks to \cite[Lemma 5.5]{GarSabTak} and \eqref{expcon} we actually have that weak solutions are uniformly bounded in $L^\infty(\Omega)\times L^\infty(\Omega)$. 
Thus, due to \cite[Theorem 1, page 1203]{Lieb} the solutions of system \eqref{sys} belong to $C^{1,\alpha}(\Omega)\times  C^{1,\alpha}(\Omega)$, for some $0<\alpha<1$.

It is important to point out that the next result is also valid for the limiting problem \eqref{a1pr} under the heterogenous assumptions established at the beginning of this work.

\begin{lemma}
\label{lemmeig}
Let $\tau_1$ be the lowest eigenvalue of the problem \eqref{sys},
 under homogeneous Dirichlet boundary conditions and denoted by 
 $$\tau_1:= \tau_1[ \mathcal{S},\Omega]$$
Moreover, $\tau_1$ is algebraically simple, isolated and it possesses a unique
positive eigenfunction (up to a multiplicative constant), denoted
by $(\varphi_{1},\psi_{1})^T$. Furthermore,  $(\varphi_{1},\psi_{1})^T$ is
strictly positive and 
 there is not any other
eigenvalue $\tau$ of \eqref{sys} satisfying
$\tau<\tau_1$.
\end{lemma}

\begin{proof}
\noindent{\bf Positivity of the first eigenvalue}. To prove the positivity of the first eigenvalue we claim that  
$$\hbox{if it exists}\quad (u,v)^T \in \mathcal{K} \setminus \{0\} \quad \hbox{so that } \quad \mathcal{S} \binom{u}{v} \geq \mu \binom{|u|^{p-2}u}{|v|^{q-2}v}, \quad \hbox{with}\quad \mu >0,$$
then, there is $\tau_1 \geq \mu$ and $(\varphi_1,\psi_1)^T \in \mathcal{K}\setminus \{0\}$ such that 
$$\mathcal{S} \binom{\varphi_1}{\psi_1} =\tau_1 \binom{|\varphi_1|^{p-2}\varphi_1}{|\psi_1|^{q-2}\psi_1}.$$
In other words, the topological degree on the positive cone $\mathcal{K}$ 
for the operator 
$$\left(\hbox{diag}(|u|^{p-2}u, |v|^{q-2}v) - \frac{1}{\mu} \mathcal{S} \binom{u}{v}\right),$$ 
in the unit ball changes. Hence, problem
\begin{equation}
\label{eql}
 \binom{|u|^{p-2}u}{|v|^{q-2}v}= \frac{1}{\mu} \mathcal{S} \binom{u}{v} +t \binom{|\varphi_1|^{p-2}\varphi_1}{|\psi|^{q-2}\psi_1},
\end{equation}
has no solution in $\mathcal{K}$ if $t>0$. 
To prove it, assume that if $(u,v)^T$ is a solution of \eqref{eql} in $\mathcal{K}\setminus\{0\}$ we have that 
$$ \binom{|u|^{p-2}u}{|v|^{q-2}v} \geq t \binom{|\varphi_1|^{p-2}\varphi_1}{|\psi|^{q-2}\psi_1} 
\quad \hbox{and} \quad \binom{|u|^{p-2}u}{|v|^{q-2}v} \geq \tilde{t} \binom{|\varphi_1|^{p-2}\varphi_1}{|\psi|^{q-2}\psi_1},$$
where $\tilde{t}=\tilde{s}^{r-1}$ is the maximum among all the $t$'s. Then,
$$ \binom{|u|^{p-2}u}{|v|^{q-2}v} \geq \frac{1}{\mu} \mathcal{S} \binom{\tilde{t} \varphi_1}{\tilde{t}\psi_1} 
+ t \binom{|\varphi_1|^{p-2}\varphi_1}{|\psi|^{q-2}\psi_1},$$
and, hence, since due to \eqref{pseu} and assuming \eqref{expcon}
$$\mathcal{S} \binom{\tilde{t} \varphi_1}{\tilde{t}\psi_1}\geq \tilde{s}^{r-1} \mathcal{S}\binom{\varphi_1}{\psi_1},$$ 
and, also, by definition of $(\varphi_1,\psi_1)^T$, we find that
\begin{align*}
\binom{|u|^{p-2}u}{|v|^{q-2}v} &  \geq \frac{\tilde{s}^{r-1} }{\mu}\mu\binom{|\varphi_1|^{p-2}\varphi_1}{|\psi|^{q-2}\psi_1} 
+ t \binom{|\varphi_1|^{p-2}\varphi_1}{|\psi|^{q-2}\psi_1}\\ & =(\tilde{t}+t) \binom{|\varphi_1|^{p-2}\varphi_1}{|\psi|^{q-2}\psi_1},
\end{align*}
which is a contradiction with the maximality of $\tilde{t}$. Therefore, there is no solution for the equation 
\eqref{eql} if $t>0$ or the topological degree 
on the positive cone $\mathcal{K}$ of the operator 
$$\left(\hbox{diag}(|u|^{p-2}u, |v|^{q-2}v) - \frac{1}{\mu} \mathcal{S} \binom{u}{v}\right),$$ 
 in the unit ball is zero if $t>0$. Consequently, $t=0$ and, hence,
$$\hbox{either, there exists a positive solution},$$
$$\hbox{or the topological degree on the positive cone $\mathcal{K}$ changes}.$$
In the first situation we are finished, we have that the eigenvalue is positive. 
On the other hand, for the second situation we have that the degree is 1 and we find a positive eigenvalue. Also, the degree is 1 if $\tau_1=0$. However, if that is the case 
we will find that if there exists $\lambda \in [0,\frac{1}{\mu}]$ such 
that, $$\binom{|u|^{p-2}u}{|v|^{q-2}v}=\lambda \mathcal{S} ({\bf w}),$$
where ${\bf w}=(u,v)^T$, with normalized ${\bf w}\in \mathcal{K}$, ${\bf w}=(u,v)^T$, then we actually have that $\lambda=\frac{1}{\mu}$. To prove it, let us take a sequence
$\lambda_n\to \frac{1}{\mu}$, as $n\to\infty$ and ${\bf w}_n$ positive solutions so that
$${\bf w}_n = \lambda_n  \mathcal{S} ({\bf w}_n), \quad \hbox{with}\quad {\bf w}_n \to {\bf w}_0,\quad \hbox{as}\quad n\to \infty,$$
where ${\bf w}_0$ is a positive solution. Thus,
$${\bf w}_n = \lambda_n  \mathcal{S} ({\bf w}_n) =\frac{1}{\mu}  \mathcal{S} ({\bf w}_n) +\left(\lambda_n-\frac{1}{\mu}\right)  \mathcal{S} ({\bf w}_n).$$
Therefore, passing to the limit we obtain only positive solutions if $\lambda=\frac{1}{\mu}$, proving that 
the first eigenvalue is strictly positive.

\noindent{\bf Non-existence of positive eigenvalue smaller (or bigger) than $\tau_1$}. Next, we prove that there is not any other positive 
eigenvalue smaller than $\tau_1$, assuming that
for the first eigenvalue its associated eigenfunction $(\varphi_1,\psi_1)^T$ has both components non-negative, i.e.
$$\binom{|\varphi_1|^{p-2}\varphi_1}{|\psi_1|^{q-2}\psi_1} = \frac{1}{\tau_1} \mathcal{S}\binom{\varphi_1}{\psi_1}, \quad 
\hbox{where}\quad (\varphi_1,\psi_1)^T \in  \mathcal{K}.$$
Subsequently, to prove that $\tau_1$ is the smallest positive eigenvalue we argue by contradiction. 
Thus, let us assume that there exists a positive eigenfunction 
$(u,v)^T$ so that $u>0$ and $v>0$ and satisfying the eigenvalue problem \eqref{sys}, i.e.
\begin{equation}
\label{eqpmax}
\binom{|u|^{p-2}u}{|v|^{q-2}v} = \frac{1}{\tau} \mathcal{S}\binom{u}{v}, \quad \hbox{where}\quad (u,v)^T \in  \mathcal{K},
\end{equation}
and with
$$0<\tau < \tau_1.$$
In other words, there is another positive eigenvalue smaller than $\tau_1$ with an associated positive eigenfunction.   
Then, by definition we have that
\begin{align*}
\binom{|u|^{p-2}u}{|v|^{q-2}v} & =\frac{1}{\tau} \mathcal{S}\binom{u}{v} =
 \frac{1}{\tau_1} \mathcal{S}\binom{u}{v}  +  \left(\frac{1}{\tau} -\frac{1}{\tau_1}\right)  \mathcal{S}\binom{u}{v}
 \\ &  = \frac{1}{\tau_1} \mathcal{S}\binom{u}{v}  +  \left(\frac{1}{\tau} -\frac{1}{\tau_1}\right)\tau  \binom{|u|^{p-2}u}{|v|^{q-2}v} \\ & \geq \frac{1}{\tau_1} \mathcal{S}\binom{u}{v} 
+\left(1 -\frac{\tau}{\tau_1}\right)^{r-1}  \binom{|u|^{p-2}u}{|v|^{q-2}v}.
\end{align*}
Note that,
$$0<1- \frac{\tau}{\tau_1}<1, \quad \hbox{since}\quad \tau<\tau_1.$$
Next, we show that for a certain $k$ (to be determined below) it follows that
\begin{equation}
\label{desop}
\mathcal{S}\binom{u}{v} \geq \mathcal{S}\binom{k\varphi_1}{k\psi_1}.
\end{equation}
To prove so we use the weak formulation of the eigenvalue problem \eqref{qeaks} i.e.
{\small \begin{align*}
\int_\Omega  |\nabla u|^{p-2}   & \nabla u \cdot \nabla \nu_1 - k^{p-1}  \int_\Omega  |\nabla \varphi_1|^{p-2}   \nabla \varphi_1\cdot \nabla \nu_1
 - \int_\Omega b\,  |u|^{\alpha}\,|v|^{\beta+1} \nu_1\\ & + k^{\alpha+\beta+1} \int_\Omega b \, |\varphi_1|^{\alpha}\,|\psi_1|^{\beta+1} \nu_1
  = \tau \int_\Omega |u|^{p-2} u \,\nu_1 -k^{p-1} \int_\Omega |\varphi_1|^{p-2} \varphi_1 \,\nu_1,
 \end{align*}
\begin{align*}
\int_\Omega  |\nabla v|^{q-2}   & \nabla v \cdot \nabla \nu_2 - k^{q-1}  \int_\Omega  |\nabla \psi_1|^{q-2}   \nabla \psi_1\cdot \nabla \nu_2
 - \int_\Omega c \, |u|^{\alpha+1}\,|v|^{\beta} \nu_2\\ & + k^{\alpha+\beta+1} \int_\Omega c\,  |\varphi_1|^{\alpha+1}\,|\psi_1|^{\beta} \nu_2
  = \tau \int_\Omega |v|^{q-2} v \,\nu_2 -k^{q-1} \int_\Omega |\psi_1|^{q-2} \psi_1\, \nu_2,
  \end{align*}}
  for a test function $(\nu_1,\nu_2)^T$. 
Since the eigenfunctions are positive and the system is cooperative ($b(x)>0$ and $c(x)>0$)  taking
$$B=\max_{x\in \overline\Omega}\{b(x),c(x)\},$$
we have that (we also add $\alpha+1$ in the first equation and $\b+1$ in the second, for the coupling terms)
\begin{align*}
\int_\Omega  |\nabla u|^{p-2}   & \nabla u \cdot \nabla \nu_1 - k^{p-1}  \int_\Omega  |\nabla \varphi_1|^{p-2}   \nabla \varphi_1 \cdot \nabla \nu_1
  - \int_\Omega b\,  |u|^{\alpha}\,|v|^{\beta+1} \nu_1  \\ & + k^{\alpha+\beta+1} \int_\Omega b \, |\varphi_1|^{\alpha}\,|\psi_1|^{\beta+1} \nu_1  \geq \int_\Omega  |\nabla u|^{p-2}  \nabla u \cdot \nabla \nu_1
   \\ & - (\alpha+1) \int_\Omega B  \, |u|^{\alpha}\,|v|^{\beta+1} \nu_1
  -k^{p-1}  \int_\Omega  |\nabla \varphi_1|^{p-2}   \nabla \varphi_1 \cdot \nabla \nu_1,  
 \end{align*}
 \begin{align*}
\int_\Omega  |\nabla v|^{q-2}   & \nabla v\cdot  \nabla \nu_2 - k^{q-1}  \int_\Omega  |\nabla \psi_1|^{q-2}   \nabla \psi_1\cdot  \nabla \nu_2
 - \int_\Omega c \, |u|^{\alpha+1}\,|v|^{\beta} \nu_2\\ & + k^{\alpha+\beta+1} \int_\Omega c\,  |\varphi_1|^{\alpha+1}\,|\psi_1|^{\beta} \nu_2 \geq \int_\Omega  |\nabla v|^{q-2}  \nabla v \cdot \nabla \nu_2
   \\ & - (\beta+1) \int_\Omega B   \,|u|^{\alpha+1}\,|v|^{\beta} \nu_2
  -k^{q-1}  \int_\Omega  |\nabla \psi_1|^{q-2}   \nabla \psi_1 \cdot \nabla \nu_2. 
 \end{align*}
Now, thanks to the variational cooperative eigenvalue problem \eqref{sysvar} and the Rayleigh quotient of the first eigenvalue $\tau_0$ \eqref{rayta} it follows that
\begin{align*}
\int_\Omega & |\nabla u|^{p-2}    \nabla u \cdot \nabla \nu_1 - k^{p-1}  \int_\Omega  |\nabla \varphi_1|^{p-2}   \nabla \varphi_1\cdot \nabla \nu_1
 - \int_\Omega b  |u|^{\alpha}\,|v|^{\beta+1} \nu_1  \\ & + k^{\alpha+\beta+1} \int_\Omega b\,  |\varphi_1|^{\alpha}\,|\psi_1|^{\beta+1} \nu_1  \geq \tau_0 \int_\Omega |u|^{p-2} u \,\nu_1  
  - k^{p-1} \int_\Omega  |\nabla \varphi_1|^{p-2}   \nabla \varphi_1\cdot  \nabla \nu_1,  
 \end{align*}
\begin{align*}
\int_\Omega  & |\nabla v|^{q-2}    \nabla v \cdot \nabla \nu_2 - k^{q-1}  \int_\Omega  |\nabla \psi_1|^{q-2}   \nabla \psi_1\cdot  \nabla \nu_2
 - \int_\Omega c  |u|^{\alpha+1}\,|v|^{\beta} \nu_2 \\ & + k^{\alpha+\beta+1} \int_\Omega c \, |\varphi_1|^{\alpha+1}\,|\psi_1|^{\beta} \nu_2 \geq \tau_0\int_\Omega |v|^{q-2}v\, \nu_2
  - k^{q-1}  \int_\Omega  |\nabla \psi_1|^{q-2}   \nabla \psi_1\cdot  \nabla \nu_2.
  \end{align*}
 Hence, in order to arrive at the inequality \eqref{desop} we find that
 \begin{equation}
 \label{kin}
 k\leq \min\left\{\left(\frac{\tau_0 \int_\Omega |u|^{p-2}u \,\nu_1}{\int_\Omega  |\nabla \varphi_1|^{p-2}  \nabla \varphi_1\cdot \nabla \nu_1}\right)^{\frac{1}{p-1}},  
 \left(\frac{\tau_0 \int_\Omega |v|^{q-2} v \,\nu_2}{ \int_\Omega  |\nabla \psi_1|^{q-2}   \nabla \psi_1\cdot \nabla \nu_2}\right)^{\frac{1}{q-1}}  \right\},
 \end{equation}
 with $u,\varphi_1,\nu_1\in W^{1,p}_0(\Omega)$ and $v,\psi_1,\nu_2\in W^{1,q}_0(\Omega)$. The maximal value for $k$ that satisfies \eqref{kin} is denoted by
$$\kappa=k_{\max}.$$
In particular, expression \eqref{kin} provides us with the condition
\begin{equation}
\label{inpmax}
( |u|^{p-2}u, |v|^{q-2}v)^T \geq \kappa^{r-1} ( |\varphi_1|^{p-2}\varphi_1, |\psi_1|^{q-2}\psi_1)^T,
\end{equation}
with the appropriate $r=p$ or $r=q$ depending on \eqref{expcon}. 
Consequently, thanks to \eqref{eqpmax} and \eqref{inpmax} we find that
\begin{align*}
\binom{|u|^{p-2}u}{|v|^{q-2}v}  & \geq \frac{1}{\tau_1} \mathcal{S}\binom{u}{v}
+\left(1 -\frac{\tau}{\tau_1}\right)^{r-1} \binom{|u|^{p-2}u}{|v|^{q-2}v}\\ & \geq  
\frac{1}{\tau_1} \mathcal{S}\binom{\kappa  \varphi_1}{\kappa \psi_1} +\left(1 -\frac{\tau}{\tau_1}\right)^{r-1}  \kappa^{r-1} \binom{|\varphi_1|^{p-2}\varphi_1}{|\psi_1|^{q-2}\psi_1}
\\ & \geq  \kappa^{r-1} \binom{|\varphi_1|^{p-2}\varphi_1}{|\psi_1|^{q-2}\psi_1} + \left(1 -\frac{\tau}{\tau_1}\right)^{r-1} \kappa^{r-1}  \binom{|\varphi_1|^{p-2}\varphi_1}{|\psi_1|^{q-2}\psi_1} 
\\ & =\left( \kappa^{r-1} +\left(1 -\frac{\tau}{\tau_1}\right)^{r-1} \kappa^{r-1}\right) \binom{|\varphi_1|^{p-2}\varphi_1}{|\psi_1|^{q-2}\psi_1},
\end{align*}
which contradicts the maximality of $\kappa^{r-1}$ and proving that there is no positive eigenvalue below $\tau_1$.

Similar arguments show that there is no eigenvalue bigger that $\tau_1$ which has a positive associated eigenfunction.

\noindent{\bf Simplicity}. To show that the principal eigenvalue $\tau_1$ is a simple eigenvalue we assume that, apart from the eigenfunction  $(\varphi_1,\psi_1)^T \in  \mathcal{K}$, 
there exists another eigenfunction $(u,v)^T\in \mathcal{K}$. Arguing as above we arrive at $(u,v)^T \geq K_1 (\varphi_1,\psi_1)^T$ and, also, 
the opposite inequality $(u,v)^T \leq K_2 (\varphi_1,\psi_1)^T$, so that $K_1K_1=1$, showing that both eigenfunctions are proportional. Therefore, the principal eigenvalue is simple.

\noindent{\bf Isolated principal eigenvalue $\tau_1$}. To prove it, we assume a sequence of eigenvalues $\{\tau_n\}$ of problem \eqref{sys} such that $\tau_n \to \tau_1$, as $n\to \infty$. 
Moreover, the sequence of eigenfunctions $\{(u_n,v_n)^T\}$ belongs to the positive cone $\mathcal{K}$, normalised $\|(u_n,v_n)^T\|_{L^\infty(\Omega\times L^\infty(\Omega)}=1$. 
Thanks to regularity theory and maximal principles we actually have that the sequence is bounded in $C^{1,\alpha}$, with $\alpha\in (0,1)$. Hence, we have the convergence of such a sequence in 
$C^{1,\alpha}$, up to a subsequence, i.e.
$$(u_n,v_n)^T \to  (\varphi_1,\psi_1)^T,\quad \hbox{as}\quad n\to \infty,$$
where $(\varphi_1,\psi_1)^T$ is the associated  eigenfunction with the principal eigenvalue $\tau_1$. Therefore, since the elements of the sequence will belong to the positive cone, and 
$\tau_1$ is the only eigenvalue with a positive eigenfunction we arrive at a contradiction, proving that such an eigenvalue is isolated. 

\end{proof}

\begin{remark}
\rm{Lemma\,\ref{lemmeig} is also true for the system with
{\small \begin{align*}
 (-\Delta_p+\lambda a |u|^{p-2}) u-b |u|^{\alpha-1}u\,|v|^\beta v & = \mathcal{A}_1(\lambda) {\bf w}, \; (-\Delta_q+\lambda d |v|^{q-2}) v- c |u|^{\alpha} u\, |v|^{\beta-1} v \\ & 
 =\mathcal{A}_2(\lambda) {\bf w}.
 \end{align*}}  
Also, thanks to the monotonicity of the principal eigenvalue with respect to the potential and with respect to the domain we know that the function 
$\tau_1(\lambda)$ is continuous and increasing with respect to the operator $\mathcal{S}(V_1,V_2)$ denoted by \eqref{a5} (cf. \cite{LeRa}) so that the limit
$$\lim_{\lambda \to \infty} \tau_1(\lambda),$$
exists. Indeed, under the spatial heterogenous conditions for the potentials $a$ and $d$ this limit is bounded above. Otherwise 
the limit could be possibly $\pm \infty$.
}
\end{remark}

Furthermore, for the limiting system we consider the problem 
$$
\begin{array}{ll}
 \left\{\begin{array}{l}
  -\Delta_p u-b |u|^{\alpha-1}u\,|v|^\beta v=\tau_1 |u|^{p-2} u,\\
  -\Delta_q v- c |u|^{\alpha} u\, |v|^{\beta-1} v=\tau_1 |v|^{q-2} v,
  \end{array}
  \right.
  &   \hbox{for}\; (u,v)^T \in W^{1,p}_0(\Omega_0^a)\times W^{1,q}_0(\Omega_0^d).
\end{array}
$$
We say that $(u,v)^T$ is a solution for the previous system when each equation is satisfied in the 
following sense/framework
$$(-\Delta_p u +\lambda a |u|^{p-2})u =f,\quad u\in W^{1,p}_0(A),$$
so that $A$ satisfies the spatial heterogeneous conditions under consideration in this paper and $f$ in the dual space $W^{-1,p'}(A)$

\section{Proof of Theorem\,\ref{stabb}}
 \label{S2}
 

Subsequently, to prove the main result of the paper Theorem\,\ref{stabb}. We will show the proof of that result as a consequence of 
several steps. First we prove the convergence of the eigenfunctions in $W^{1,p}_0(\Omega)\times W^{1,q}_0(\Omega)$.

\vspace{0.2cm}

\noindent{\bf Step 1:} Convergence of the eigenfunctions $\{(\varphi_\lambda,\psi_\lambda)^T\}$ in $X:=W^{1,p}_0(\Omega)\times W^{1,q}_0(\Omega)$. 

\vspace{0.2cm}

\noindent Let $\{\l_n\}_{n\geq 1}$ be any increasing unbounded sequence,
i.e., so that $0<\lambda_n<\lambda_m$ if $n<m$ and
\begin{equation*}
  \lim_{n\to \infty}\l_n=\infty.
\end{equation*}
Then, for every $n\geq 1$
we consider a sequence $\{(\varphi_{\lambda_n},\psi_{\lambda_n})^T\}$ of normalised solutions in $Y:=L^p(\Omega)\times L^q(\Omega)$ for the system \eqref{a1} in the sense that
$$\int_\Omega |\varphi_{\lambda_n}|^p + \int_\Omega |\psi_{\lambda_n}|^q =1,$$ 
associated with $\tau_1(\lambda_n)$ for the problem \eqref{a1}. 
Then, multiplying \eqref{a1} by $(\varphi_{\lambda_n},\psi_{\lambda_n})^T$
and integrating by parts yields
\begin{align*}
\int_\Omega |\nabla \varphi_{\lambda_n}|^p + 
	\int_\Omega  |\nabla \psi_{\lambda_n}|^q   + \lambda_n \int_\Omega (a |\varphi_{\lambda_n}|^p + d |\psi_{\lambda_n}|^q)
	& =\tau_1(\lambda_n) \\ & + \int_\Omega (b+c) |\varphi_{\lambda_n}|^{\alpha+1} \, |\psi_{\lambda_n} |^{\beta+1}.
\end{align*}
It is now clear that by H\"{o}lder's inequality, condition \eqref{pq} and the cooperative assumptions on the coefficients $b$ and $c$ we find that
$$ \int_\Omega (b+c) |\varphi_{\lambda_n}|^{\alpha+1} \, |\psi_{\lambda_n} |^{\beta+1} \leq C \, \left(  \int_\Omega  |\varphi_{\lambda_n}|^{p} \right)^{\frac{\alpha+1}{p}}  
\left(  \int_\Omega  |\psi_{\lambda_n}|^{q} \right)^{\frac{\beta+1}{q}},$$
for a positive constant $C>0$, so that, by construction we have 
\begin{equation}
\label{inbou}
	\int_\Omega |\nabla \varphi_{\lambda_n}|^p \leq K, \quad
	\int_\Omega  |\nabla \psi_{\lambda_n}|^q \leq K, \quad \lambda_n \int_\Omega (a \varphi_{\lambda_n}^p + d \psi_{\lambda_n}^q)\leq K,
\end{equation}
 for a positive constant $K$. We point out that the principal eigenvalue $\tau_1(\lambda)$ is bounded above, thanks to the monotonicity of the principal eigenvalue with respect to 
 the domain, by the principal eigenvalue for the operator $\mathcal{S}$ under Dirichlet homogeneous boundary conditions in the subdomain $\Omega_0$ and denoted by \eqref{simop}, i.e.
 $$\tau_1(\lambda) \leq \tau_1[\mathcal{S};\Omega_0].$$ 
 Hence, the sequence $\{(\varphi_{\lambda_n},\psi_{\lambda_n})^T\}$ is bounded in $X:=W^{1,p}_0(\Omega)\times W^{1,q}_0(\Omega)$.  
 As the imbedding
\begin{equation*}
  W_0^{1,r}(\Omega)\hookrightarrow L^r(\Omega),\quad \hbox{with}\quad r=p,q,
\end{equation*}
is compact, 
 we can extract a subsequence, again
labelled by $\{(\varphi_{\lambda_n},\psi_{\lambda_n})^T\}$, weakly convergent
in $X$ and strongly in $Y$
to some function $(\varphi_*,\psi_*)\in Y$, i.e.
$$\lim_{\lambda \to \infty} \|(\varphi_{\lambda_n} ,\psi_{\lambda_n})^T- (\varphi_*,\psi_*)^T\|_{Y}=0,\quad Y=L^p(\Omega)\times L^q(\Omega).$$
In fact, we will prove in the sequel that the sequence $\{(\varphi_{\lambda_n},\psi_{\lambda_n})^T\}$ is actually a Cauchy sequence in $X$. In other words, 
we have the strong convergence of that subsequence in $U$, this implies that
$$\lim_{\lambda \to \infty} \|(\varphi_{\lambda_n} ,\psi_{\lambda_n})^T- (\varphi_*,\psi_*)^T\|_{X}=0,\quad X=W^{1,p}_0(\O)\times W^{1,q}_0(\O).$$
with $(\varphi_*,\psi_*)^T\in X$ and such that 
$$\int_\Omega |\varphi_*|^p+ \int_\Omega |\psi_*|^q=1.$$
Subsequently, fix $n<m$ so that $0<\lambda_n<\lambda_m$ and set
\begin{equation}
\label{38}
   D_{n,m}  :=  \int_{\Omega} \left|\nabla (\varphi_{\lambda_n} - \varphi_{\lambda_m})\right|^p
    +  \int_{\Omega} \left|\nabla (\psi_{\lambda_n} -
    \psi_{\lambda_m})\right|^q.
\end{equation}
In order to get the convergence we consider the so-called Clarkson's inequality i.e.
\begin{equation}
\label{clark}
\frac{|u_1 - u_2|^p}{2^{p-1}} + \frac{|u_1 + u_2|^p}{2^{p-1}}\leq |u_1|^p + |u_2|^p.
\end{equation}
Also, thanks to the strict convexity of the mapping $u \mapsto |u|^p$ it follows that
\begin{equation}
\label{convex}
|u_2|^p > |u_1|^p+p |u_1|^{p-2} u_1(u_2-u_1),\quad \hbox{for points in $\mathbb{R}^N$ and $u_1\neq u_2$, $p>1$.}
\end{equation}
Hence, according to \eqref{convex} for $u_1+u_2$
\begin{equation}
\label{con22}
\frac{|u_1 + u_2|^p}{2^{p}} \geq |u_1|^p+\frac{1}{2} p |u_1|^{p-2} u_1(u_2-u_1).
\end{equation}
Consequently, combining both inequalities \eqref{clark}, \eqref{con22} we find that
$$\frac{|u_1 - u_2|^p}{2^{p-1}} \leq |u_2|^p - |u_1|^p-p |u_1|^{p-2} u_1(u_2-u_1),$$
and, hence, if $\nabla (\varphi_{\lambda_n} - \varphi_{\lambda_m})= u_1-u_2$ it yields 
$$\frac{|\nabla \varphi_{\lambda_n} - \nabla \varphi_{\lambda_m} |^p}{2^{p-1}} \leq |\nabla \varphi_{\lambda_m} |^p - |\nabla \varphi_{\lambda_n} |^p-p 
|\nabla \varphi_{\lambda_n} |^{p-2} \nabla \varphi_{\lambda_n} \cdot(\nabla \varphi_{\lambda_m} -\nabla \varphi_{\lambda_n}).$$
Indeed, integrating in $\Omega$ 
\begin{align*}
\frac{1}{2^{p-1}}  \int_{\Omega} |\nabla \varphi_{\lambda_n} - \nabla \varphi_{\lambda_m} |^p  & \leq  \int_{\Omega}|\nabla \varphi_{\lambda_m} |^p 
\\ & -  \int_{\Omega} |\nabla \varphi_{\lambda_n} |^p- \int_{\Omega} p 
|\nabla \varphi_{\lambda_n} |^{p-2} \nabla \varphi_{\lambda_n} \cdot(\nabla \varphi_{\lambda_m} -\nabla \varphi_{\lambda_n}).
\end{align*}
Then, 
\begin{align*}
D_{n,m}  & \leq \int_{\Omega}|\nabla \varphi_{\lambda_m} |^p -  \int_{\Omega} |\nabla \varphi_{\lambda_n} |^p- \int_{\Omega} p 
|\nabla \varphi_{\lambda_n} |^{p-2} \nabla \varphi_{\lambda_n} \cdot(\nabla \varphi_{\lambda_m} -\nabla \varphi_{\lambda_n})
\\ & +\int_{\Omega}|\nabla \psi_{\lambda_m} |^q -  \int_{\Omega} |\nabla \psi_{\lambda_n} |^q- \int_{\Omega} q 
|\nabla \psi_{\lambda_n} |^{q-2} \nabla \psi_{\lambda_n} \cdot(\nabla \psi_{\lambda_m} -\nabla \psi_{\lambda_n}).
\end{align*}
Now, applying the weak formulation \eqref{sisin} for the eigenvalue problem \eqref{a1} with the principal eigenfunctions $\{(\varphi_{\lambda_n},\psi_{\lambda_n})^T\}$ 
(respectively $\{(\varphi_{\lambda_m},\psi_{\lambda_m})^T\}$) 
and rearranging terms we find that
{\small \begin{align*}
D_{n,m}  & \leq  \tau_1(\lambda_m) \int_\Omega |\varphi_{\lambda_m}|^{p}- \lambda_m \int_\Omega a |\varphi_{\lambda_m}|^{p}
+ \int_\Omega b |\varphi_{\lambda_m}|^{\alpha+1} |\psi_{\lambda_m}|^{\beta+1} \\ & 
+ \tau_1(\lambda_m) \int_\Omega |\psi_{\lambda_m}|^{q}- \lambda_m \int_\Omega d |\psi_{\lambda_m}|^{q}+ \int_\Omega c |\varphi_{\lambda_m}|^{\alpha+1} |\psi_{\lambda_m}|^{\beta+1} \\ & 
+ (p-1)\left[ \tau_1(\lambda_n) \int_\Omega |\varphi_{\lambda_n}|^{p}- \lambda_n \int_\Omega a |\varphi_{\lambda_n}|^{p}
+ \int_\Omega b |\varphi_{\lambda_n}|^{\alpha+1} |\psi_{\lambda_n}|^{\beta+1}\right] \\ & 
+ (q-1)\left[\tau_1(\lambda_n) \int_\Omega |\psi_{\lambda_n}|^{q}- \lambda_n \int_\Omega d |\psi_{\lambda_n}|^{q}+ \int_\Omega c |\varphi_{\lambda_n}|^{\alpha+1} |\psi_{\lambda_n}|^{\beta+1}\right] \\ & 
-p \left[ \tau_1(\lambda_n) \int_\Omega |\varphi_{\lambda_n}|^{p-1} \varphi_{\lambda_m}- \lambda_n \int_\Omega a |\varphi_{\lambda_n}|^{p-1} \varphi_{\lambda_m}
+ \int_\Omega b |\varphi_{\lambda_n}|^{\alpha} |\psi_{\lambda_n}|^{\beta+1} \varphi_{\lambda_m}\right]\\ &  
-q \left[ \tau_1(\lambda_n) \int_\Omega |\psi_{\lambda_n}|^{q-1} \psi_{\lambda_m}- \lambda_n \int_\Omega a |\psi_{\lambda_n}|^{q-1} \psi_{\lambda_m}
+ \int_\Omega c |\varphi_{\lambda_n}|^{\alpha+1} |\psi_{\lambda_n}|^{\beta} \psi_{\lambda_m}\right].
\end{align*}}
Thus, 
{\small \begin{align*}
D_{n,m}  & \leq  \tau_1(\lambda_m) \left(\int_\Omega |\varphi_{\lambda_m}|^{p} + \int_\Omega |\psi_{\lambda_m}|^{q} \right)
- \tau_1(\lambda_n) \left(\int_\Omega |\varphi_{\lambda_n}|^{p} + \int_\Omega |\psi_{\lambda_n}|^{q} \right)\\ & 
+p \tau_1(\lambda_n) \int_\Omega |\varphi_{\lambda_n}|^{p-1} (\varphi_{\lambda_n}-\varphi_{\lambda_m}) +q \tau_1(\lambda_n) \int_\Omega |\psi_{\lambda_n}|^{q-1} (\psi_{\lambda_n}-\psi_{\lambda_m})\\ & 
+ p \lambda_n  \int_\Omega a |\varphi_{\lambda_n}|^{p-1} (\varphi_{\lambda_m}-\varphi_{\lambda_n})
+ q \lambda_n  \int_\Omega d |\psi_{\lambda_n}|^{q-1} (\psi_{\lambda_m}-\psi_{\lambda_n})\\ & 
+ \lambda_n  \int_\Omega a |\varphi_{\lambda_n}|^{p} - \lambda_m \int_\Omega a |\varphi_{\lambda_m}|^{p} 
+ \lambda_n \int_\Omega d |\psi_{\lambda_n}|^{q} -  \lambda_m \int_\Omega d |\psi_{\lambda_m}|^{q}\\ & 
+ p  \int_\Omega b |\varphi_{\lambda_n}|^{\alpha} |\psi_{\lambda_n}|^{\beta+1} (\varphi_{\lambda_n}-\varphi_{\lambda_m}) 
+ q \int_\Omega c |\varphi_{\lambda_n}|^{\alpha+1} |\psi_{\lambda_n}|^{\beta} (\psi_{\lambda_n}-\psi_{\lambda_m}).
\end{align*}}
For the terms involving the non-negative potentials $a$ and $d$ and thanks to the convexity property \eqref{convex}, it follows that,
{\small $$ \lambda_n  \int_\Omega a |\varphi_{\lambda_n}|^{p} 
+ p \lambda_n  \int_\Omega a |\varphi_{\lambda_n}|^{p-1} (\varphi_{\lambda_m}-\varphi_{\lambda_n})- \lambda_m \int_\Omega a |\varphi_{\lambda_m}|^{p} < 
(\lambda_n-\lambda_m) \int_\Omega a |\varphi_{\lambda_m}|^{p} \leq 0,$$}
and
{\small $$ \lambda_n \int_\Omega d |\psi_{\lambda_n}|^{q} + q \lambda_n  \int_\Omega d |\psi_{\lambda_n}|^{q-1} (\psi_{\lambda_m}-\psi_{\lambda_n}) -  \lambda_m \int_\Omega d |\psi_{\lambda_m}|^{q} < 
(\lambda_n-\lambda_m) \int_\Omega d |\psi_{\lambda_m}|^{q} \leq 0,$$}
since, by construction $n<m$ and, then, $\lambda_n<\lambda_m$. 
Hence, after adding and subtracting some appropriate terms it yields
\begin{align*}
D_{n,m}  & \leq  (\tau_1(\lambda_m)-\tau_1(\lambda_n)) \left(\int_\Omega |\varphi_{\lambda_m}|^{p} + \int_\Omega |\psi_{\lambda_m}|^{q} \right)
\\ & + \tau_1(\lambda_n) \left(\int_\Omega (|\varphi_{\lambda_m}|^{p}- |\varphi_{\lambda_n}|^{p}  
+ \int_\Omega (|\psi_{\lambda_m}|^{q}- |\psi_{\lambda_n}|^{q}) \right)\\ & 
+p \tau_1(\lambda_n) \int_\Omega |\varphi_{\lambda_n}|^{p-1} (\varphi_{\lambda_n}-\varphi_{\lambda_m}) 
+q \tau_1(\lambda_n) \int_\Omega |\psi_{\lambda_n}|^{q-1} (\psi_{\lambda_n}-\psi_{\lambda_m})\\ & 
+ p  \int_\Omega b |\varphi_{\lambda_n}|^{\alpha} |\psi_{\lambda_n}|^{\beta+1} (\varphi_{\lambda_n}-\varphi_{\lambda_m}) 
+ q \int_\Omega c |\varphi_{\lambda_n}|^{\alpha+1} |\psi_{\lambda_n}|^{\beta} (\psi_{\lambda_n}-\psi_{\lambda_m}).
\end{align*}
For the second term we use again the convexity property \eqref{convex}, thus
$$\tau_1(\lambda_n) \left(\int_\Omega (|\varphi_{\lambda_m}|^{p}- |\varphi_{\lambda_n}|^{p}) \right)  
\leq\tau_1(\lambda_n)  p \int_\Omega |\varphi_{\lambda_m}|^{p-2} \varphi_{\lambda_m}(\varphi_{\lambda_m}-\varphi_{\lambda_n}),$$
$$\tau_1(\lambda_n) \left(\int_\Omega (|\psi_{\lambda_m}|^{q}- |\psi_{\lambda_n}|^{q}) \right) \leq \tau_1(\lambda_n)  
q \int_\Omega |\psi_{\lambda_m}|^{q-2} \psi_{\lambda_m}(\psi_{\lambda_m}-\psi_{\lambda_n}).$$
Consequently, applying H\"{o}lder's inequality and the fact that the eigenfunctions are bounded in $Y=L^p(\Omega)\times L^q(\Omega)$ to the different terms
we have that for a constant $C>0$
$$
 (\tau_1(\lambda_m)-\tau_1(\lambda_n)) \left(\int_\Omega |\varphi_{\lambda_m}|^{p} + \int_\Omega |\psi_{\lambda_m}|^{q} \right)\leq C (\tau_1(\lambda_m)-\tau_1(\lambda_n)),
 $$
\begin{align*}  
\tau_1(\lambda_n) \int_\Omega (|\varphi_{\lambda_m}|^{p}- |\varphi_{\lambda_n}|^{p}) & \leq \tau_1(\lambda_n)  
  p \int_\Omega |\varphi_{\lambda_m}|^{p-2} \varphi_{\lambda_m}(\varphi_{\lambda_m}-\varphi_{\lambda_n}) \\ & \leq 
  C \|\varphi_{\lambda_m} -\varphi_{\lambda_n}\|_{L^p(\Omega)},
  \end{align*}
\begin{align*}   
\tau_1(\lambda_n)  \int_\Omega (|\psi_{\lambda_m}|^{q}- |\psi_{\lambda_n}|^{q}) & \leq  \tau_1(\lambda_n)  q \int_\Omega |\psi_{\lambda_m}|^{q-2} \psi_{\lambda_m}(\psi_{\lambda_m}-\psi_{\lambda_n})
\\ & \leq 
  C \|\psi_{\lambda_m}- \psi_{\lambda_n}\|_{L^q(\Omega)},
  \end{align*} 
\begin{align*}  
p \tau_1(\lambda_n) \int_\Omega |\varphi_{\lambda_n}|^{p-1} (\varphi_{\lambda_n}-\varphi_{\lambda_m}) &  \leq C 
  \left( \int_\Omega |\varphi_{\lambda_n}|^{p'}\right)^{1/p'} \left(\int_\Omega |\varphi_{\lambda_n}-\varphi_{\lambda_m}|^p\right)^{1/p}
  \\ & \leq C \|\varphi_{\lambda_n}-\varphi_{\lambda_m}\|_{L^p(\Omega)},
  \end{align*}
\begin{align*}
q \tau_1(\lambda_n) \int_\Omega |\psi_{\lambda_n}|^{q-1} (\psi_{\lambda_n}-\psi_{\lambda_m}) & \leq C 
  \left( \int_\Omega |\psi_{\lambda_n}|^{q'}\right)^{1/q'} \left(\int_\Omega |\psi_{\lambda_n}-\psi_{\lambda_m}|^q\right)^{1/q} \\ & \leq C  \|\psi_{\lambda_n}-\psi_{\lambda_m}\|_{L^q(\Omega)}.
\end{align*}
Also, for the terms with the cooperative coefficients $b$ and $c$ 
$$
  p  \int_\Omega b |\varphi_{\lambda_n}|^{\alpha} |\psi_{\lambda_n}|^{\beta+1} (\varphi_{\lambda_n}-\varphi_{\lambda_m})\quad 
  \hbox{and}\quad q \int_\Omega c |\varphi_{\lambda_n}|^{\alpha+1} |\psi_{\lambda_n}|^{\beta} (\psi_{\lambda_n}-\psi_{\lambda_m}).
$$
we will apply Young's and H\"{o}lder's inequalities. Indeed, due to Young's inequality and assuming the cooperative term $b(x)$ is bounded in $\Omega$ we find that
{\small $$
 p  \int_\Omega b |\varphi_{\lambda_n}|^{\alpha} |\psi_{\lambda_n}|^{\beta+1} (\varphi_{\lambda_n}-\varphi_{\lambda_m})
 \leq C \int_\Omega \left(\frac{\alpha+1}{p} |\varphi_{\lambda_n}|^{\frac{\alpha p}{\alpha+1}} + \frac{\beta+1}{q} |\psi_{\lambda_n}|^q\right)(\varphi_{\lambda_n}-\varphi_{\lambda_m}).
$$}
Moreover, thanks to H\"{o}lder's inequality 
{\tiny \begin{align*}
\int_\Omega \left(\frac{\alpha+1}{p} |\varphi_{\lambda_n}|^{\frac{\alpha p}{\alpha+1}} + \frac{\beta+1}{q} |\psi_{\lambda_n}|^q\right)
(\varphi_{\lambda_n}-\varphi_{\lambda_m}) &  \leq C \left(\int_\Omega |\varphi_{\lambda_n}|^{\frac{\alpha (p-1)}{\alpha+1}}\right)^{1/p'}  
\left(\int_\Omega |\varphi_{\lambda_n}-\varphi_{\lambda_m}|^p\right)^{1/p} \\ & 
+  C \left(  \int_\Omega  |\psi_{\lambda_n}|^{q p'} \right)^{1/p'} 
\left(\int_\Omega |\varphi_{\lambda_n}-\varphi_{\lambda_m}|^p\right)^{1/p}.
\end{align*}}
Thus, according \cite[Lemma 2]{CuQu} it follows that
$$\|\psi_{\lambda_n}\|_{qp'}^q \leq \epsilon \int_\Omega |\nabla \psi_{\lambda_n}|^q + M_\epsilon \int_\Omega \omega |\psi_{\lambda_n}|^q,$$
for $\epsilon>0$, a positive constant $M_\epsilon$ depending on $\epsilon$ and a bounded positive weight $\omega$. Then, we finally have that 
\begin{align*}
p  \int_\Omega b |\varphi_{\lambda_n}|^{\alpha} |\psi_{\lambda_n}|^{\beta+1} & (\varphi_{\lambda_n}-\varphi_{\lambda_m}) 
 \leq C \left(\int_\Omega |\varphi_{\lambda_n}|^{\frac{\alpha (p-1)}{\alpha+1}}\right)^{1/p'} \|\varphi_{\lambda_n}-\varphi_{\lambda_m}\|_{L^p(\Omega)}
\\ & + C \left( \epsilon \int_\Omega |\nabla \psi_{\lambda_n}|^q + M_\epsilon \int_\Omega \omega |\psi_{\lambda_n}|^q\right) \|\varphi_{\lambda_n}-\varphi_{\lambda_m}\|_{L^p(\Omega)}.
\end{align*}
 Similarly, for the term with the cooperative coefficient $c$ we find that
\begin{align*}
 q \int_\Omega c |\varphi_{\lambda_n}|^{\alpha+1} |\psi_{\lambda_n}|^{\beta} & (\psi_{\lambda_n}-\psi_{\lambda_m})  \leq C 
 \left(\int_\Omega |\psi_{\lambda_n}|^{\frac{\beta (q-1)}{\beta+1}}\right)^{1/q'}  
\left(\int_\Omega |\psi_{\lambda_n}-\psi_{\lambda_m}|^q\right)^{1/q} \\ & 
+  C \left(   \epsilon \int_\Omega |\nabla \varphi_{\lambda_n}|^p + M_\epsilon \int_\Omega \omega |\varphi_{\lambda_n}|^p \right)
\left(\int_\Omega |\psi_{\lambda_n}-\psi_{\lambda_m}|^q\right)^{1/q},
\end{align*}
supposing that 
$$\|\varphi_{\lambda_n}\|_{pq'}^p \leq \epsilon \int_\Omega |\nabla \varphi_{\lambda_n}|^p + M_\epsilon \int_\Omega \omega |\varphi_{\lambda_n}|^p,$$
Hence, since the eigenfunctions are bounded in $Y=L^p(\Omega)\times L^q(\Omega)$, 
$$\frac{\alpha (p-1)}{\alpha+1}<p \quad \hbox{and}\quad \frac{\beta(q-1)}{\beta+1}<q,$$
and thanks to \eqref{inbou} we finally obtain that 
\begin{align*}
 p  \int_\Omega b |\varphi_{\lambda_n}|^{\alpha} |\psi_{\lambda_n}|^{\beta+1} (\varphi_{\lambda_n}-\varphi_{\lambda_m}) & \leq C \|\varphi_{\lambda_n}-\varphi_{\lambda_m}\|_{L^p(\Omega)},\\ 
 q \int_\Omega c |\varphi_{\lambda_n}|^{\alpha+1} |\psi_{\lambda_n}|^{\beta} (\psi_{\lambda_n}-\psi_{\lambda_m})  & \leq C \|\psi_{\lambda_n}-\psi_{\lambda_m}\|_{L^q(\Omega)}.
 \end{align*}
Therefore, according to the previous inequalities we find that there exists a positive constant $C>0$ so that 
\begin{align*}
\int_{\Omega}  & \left|\nabla (\varphi_{\lambda_n} - \varphi_{\lambda_m})\right|^p
    + \int_{\Omega} \left|\nabla   (\psi_{\lambda_n} -
    \psi_{\lambda_m})\right|^q   \leq C (\tau_1(\lambda_m)-\tau_1(\lambda_n)) \\ & + C  \|\varphi_{\lambda_n}-\varphi_{\lambda_m}\|_{L^p(\Omega)} + C  \|\psi_{\lambda_n}-\psi_{\lambda_m}\|_{L^q(\Omega)}.
    \end{align*}
Thus, thanks to the convergence of the sequence $\{\varphi_{\lambda_n}\}$ in $L^p$, the sequence $\{\psi_{\lambda_n}\}$ in $L^q$ and the fact that the function $\tau_1(\lambda)$ is convergent,
since it is an increasing function in $\lambda$ and bounded above, by the monotonicity of the principal eigenvalue with respect to the potential and the domain, 
we actually have that the sequence $\{(\varphi_{\lambda_n},\psi_{\lambda_n})^T\}$ is a Cauchy sequence in $X$ so that the limit $(\varphi_*,\psi_*)^T$ satisfies that
$$(\varphi_*,\psi_*) \geq (0,0),\quad \hbox{and}\quad \int_\Omega |\varphi_*|^p+ \int_\Omega |\psi_*|^q=1.$$

\vspace{0.2cm}

\noindent{\bf Step 2:} $\{(\varphi_*,\psi_*)^T\}$ belongs to the functional space $W^{1,p}_0(\Omega_0^{a})\times W^{1,q}_0(\Omega_0^{d})$. 

\vspace{0.2cm}

\noindent First thanks to \eqref{inbou} we have that the sequence $\{(\varphi_{\lambda_n},\psi_{\lambda_n})^T\}$ is bounded in $X$. Then, 
as performed above we can extract a subsequence, again
labelled by $\{(\varphi_{\lambda_n},\psi_{\lambda_n})^T\}$, weakly convergent
in $X$ and strongly in $L^p(\Omega) \times L^q(\Omega)$
to some function $(\varphi_*,\psi_*)^T\in L^p(\Omega) \times L^q(\Omega)$. Actually, we have proved that the sequence converges strongly in 
$W_0^{1,p}(\Omega) \times W_0^{1,q}(\Omega)$.

Furthermore, according to \eqref{inbou} it follows 
$$\lambda_n \int_\Omega a \varphi_{\lambda_n}^p\leq K\quad \hbox{and}\quad  \lambda_n \int_\Omega d \psi_{\lambda_n}^q\leq K.$$
Then, since $\{(\varphi_{\lambda_n},\psi_{\lambda_n})^T\}$ converges strongly in $L^p(\Omega) \times L^q(\Omega)$ 
to $(\varphi_*,\psi_*)^T\in L^p(\Omega) \times L^q(\Omega)$. Indeed, 
$$\lim_{n\to \infty} \int_\Omega a \varphi_{\lambda_n}^p = 0\quad \hbox{and}\quad \lim_{n\to \infty} \int_\Omega d \psi_{\lambda_n}^q=0,$$
from which we easily deduce that 
$$\varphi_* = 0 \quad \hbox{a.e. in $\Omega_+^{a}$ and},\quad \psi_* = 0 \quad\hbox{a.e. in $\Omega_+^{d}$}.$$
Consequently, we can conclude that 
$$(\varphi_*,\psi_*)^T\in W_0^{1,p}(\Omega_0^{a}) \times W_0^{1,q}(\Omega_0^{d}).$$

\vspace{0.2cm}

\noindent{\bf Step 3:} $(\varphi_*,\psi_*)^T$ eigenfunction of the limiting problem in $W^{1,p}_0(\Omega_0^{a})\times W^{1,q}_0(\Omega_0^{d})$. 

\vspace{0.2cm}

\noindent As a consequence of Step 1 it implies that 
$$\hspace{1cm} \tau_1(\lambda_n) |\varphi_{\lambda_n}|^{p-2} \varphi_{\lambda_n}  + b |\varphi_{\lambda_n} |^{\alpha-1}\varphi_{\lambda_n} \,|\psi_{\lambda_n} |^\beta \psi_{\lambda_n} \quad 
\hbox{and}$$ 
$$\tau_1(\lambda_n) |\psi_{\lambda_n}|^{q-2} \psi_{\lambda_n} +c |\varphi_{\lambda_n}|^{\alpha} \varphi_{\lambda_n} \, |\psi_{\lambda_n} |^{\beta-1} \psi_{\lambda_n},$$
converges strongly in $L^p(\Omega) \times L^q(\Omega)$ to
$$\tau_1 |\varphi_*|^{p-2}\varphi_*  + b |\varphi_*|^{\alpha-1}\varphi_*\,|\psi_*|^\beta \psi_* \quad \hbox{and} \quad \tau_1 |\psi_*|^{q-2} \psi_* +c |\varphi_*|^{\alpha} \varphi_*\, |\psi|^{\beta-1} \psi_*,$$
 respectively.
 Therefore, from the assumptions on the potentials $a$ and $d$ and Step 2 we find that 
 up to a subsequence, $\{(\varphi_{\lambda_n},\psi_{\lambda_n})^T\}$ converges strongly in $W_0^{1,p}(\Omega) \times W_0^{1,q}(\Omega)$ and the limit 
 $$(\varphi_*,\psi_*)^T \in W_0^{1,p}(\Omega_0^a) \times  W_0^{1,q}(\Omega_0^d),$$ 
 is a solution of the eigenvalue problem 
\begin{equation}
\label{a2lim}
\left\{
\begin{array}{l}
  -\Delta_p \varphi_*-b |\varphi|^{\alpha-1}\varphi_*\,| \mathcal{P}\psi_*|^\beta \mathcal{P}\psi_*=\tau_1 \varphi_*\\
  -\Delta_q \psi_*- c |\mathcal{P}\varphi_*|^{\alpha} \mathcal{P}\varphi_*\, |\psi_*|^{\beta-1} \psi_*=\tau_1  \psi_*
  \end{array}
    \right.
\end{equation}
for $(\varphi_*,\psi_*)^T \in W_0^{1,p}(\Omega_0^a) \times  W_0^{1,q}(\Omega_0^d)$ and 
where $\mathcal{P}$ stands for the projection on the subdomain where both potentials vanish at the same time, i.e.
$$\mathcal{P}\phi=\chi_{\Omega_0^{a}\cap \Omega_0^{d}}\phi, \quad \hbox{with}\quad \phi=\varphi_*,\psi_*.$$
Thus, $\mathcal{P}$ is defined to be zero if the intersection is null
$$\Omega_0^{a}\cap \Omega_0^{d}=\emptyset.$$
Furthermore, by the uniqueness of the principal eigenvalue of a singular $p$-Laplacian equation the sequence converges to the eigenfunction associated with the principal eigenvalue
$\tau_1$ for the uncoupled system \eqref{uncop}, i.e. the principal eigenfunction whose components are the corresponding eigenfunctions for each equation in \eqref{uncop}.
Actually, we observe that if that was the case the limiting eigenvalue would be the infimum between the corresponding eigenvalues of the following uncoupled 
system
$$
\left\{
\begin{array}{ll}
  -\Delta_p \varphi_*=\tau_1 |\varphi_*|^{p-2} \varphi_* & \hbox{in}\quad \Omega_0^{a},\\
  -\Delta_q \psi_*=\tau_1 |\psi_*|^{q-2} \psi_* & \hbox{in}\quad \Omega_0^{d},
  \end{array}
    \right.$$
    under homogeneous boundary conditions and such that
    $$\tau_1=\inf\{\tau_1[ -\Delta_p,\Omega_0^{a}],\tau_1[ -\Delta_q,\Omega_0^{d}]\}.$$


\end{document}